\documentclass[11pt,letterpaper]{amsart}
\usepackage{amssymb,latexsym,amsmath, graphicx} 
\usepackage{enumerate,enumitem,color,hyperref} 
\usepackage{mathrsfs} 
\usepackage{calligra}
\usepackage[T1]{fontenc}
\usepackage{tikz}
\usepackage{tikz-cd}

\usepackage{fullpage}
\input{xypic}
\xyoption{all}

\title[\bf Cuspidal cohomology]{Cuspidal cohomology of ${\rm GL}(n)$ over a number field}

\author{ \bf Nasit Darshan \ \ \& \ \ A. Raghuram} 

\date{\today}      
\subjclass[2020]{11F75; 11F70, 22E41, 22E55}

\address{Dept.\,of Mathematics, Indian Institute of Science Education and Research, Dr.\,Homi Bhabha Road, Pune 411008, INDIA.}

\email{nasit.darshan@students.iiserpune.ac.in}

\address{Dept.\,of Mathematics, Fordham University at Lincoln Center, New York, NY 10023, USA.} 

\email{araghuram@fordham.edu}

\numberwithin{equation}{section}   
\newtheorem{lemma}[equation]{Lemma}

\newtheorem{thm}[equation]{Theorem}
\newtheorem{prop}[equation]{Proposition}

\newtheorem{con}[equation]{Conjecture}
\newtheorem{claim}[equation]{Claim}

\usepackage{bm}
\usepackage{upgreek}
\makeatletter
\newcommand{\bfgreek}[1]{\bm{\@nameuse{up#1}}}

\def\blambda{\bfgreek{lambda}}

\let\oldtocsection=\tocsection
\let\oldtocsubsection=\tocsubsection
\let\oldtocsubsubsection=\tocsubsubsection
\renewcommand{\tocsection}[2]{\hspace{0em}\oldtocsection{#1}{#2}}
\renewcommand{\tocsubsection}[2]{\hspace{1em}\oldtocsubsection{#1}{#2}}
\renewcommand{\tocsubsubsection}[2]{\hspace{2em}\oldtocsubsubsection{#1}{#2}}

\begin{document}

 
\begin{abstract}
The main result of this article proves the nonvanishing of cuspidal cohomology for ${\rm GL}(n)$ over a number field which is Galois over its maximal totally real subfield. 
The proof uses the internal structure of a strongly-pure weight that can possibly support cuspidal cohomology 
and the foundational work of Borel, Labesse, and Schwermer. 
\end{abstract} 

\maketitle

 \def\R{\mathbb{R}}
\def\C{\mathbb{C}}
\def\Z{\mathbb{Z}}
\def\Q{\mathbb{Q}}
\def\A{\mathbb{A}}
\def\F{\mathbb{F}}
 \def\J{\mathbb{J}}
\newcommand\D{{\mathbb{ D}}}
 \def\L{\mathbb{L}}
\def\bG{\mathbb{G}}
\def\N{\mathbb{N}}
\def\BH{\mathbb{H}}
\newcommand\Qi{\Q(\i)}

\newcommand{\vless}{\rotatebox[origin=c]{-90}{$<$}}
\newcommand{\vgreat}{\rotatebox[origin=c]{90}{$<$}}
\newcommand{\vgreater}{\rotatebox[origin=c]{90}{$\leq$}}

\newcommand\Dm{\D_\lambda}
\newcommand\Dmp{\D_{\lambda^\prime}}
 \newcommand\Dim{\D_{\io\lambda}}
 \newcommand\Dimp{\D_{\io\lambda^\prime}}
\newcommand\Dum{\D_{\ul{\lambda}}}
\newcommand\Dium{\D_{\io\ul{\lambda}}} 
\newcommand\DumN{\D_{\ul{\lambda}-N\gamma_P}} 
\newcommand\DiumN{\D_{\io\ul{\lambda}-N\gamma_P}} 
\newcommand\cal{\mathcal}
\newcommand\SMK{{\cal S}^M_{K^M_f}}
\newcommand\tMZl{\tM_{\lambda,\Z}} 
\newcommand\Gm{{\mathbb G}_m}
\newcommand\cA{\mathcal{A}}
\newcommand\cC{\mathcal{C}}
\newcommand\calL{\mathcal{L}}
\newcommand\cO{\mathcal{O}}
\newcommand\cU{\mathcal{U}}
\newcommand\cK{\mathcal{K}}   
\newcommand\cW{\mathcal{W}}
\newcommand\HH{\mathcal{H}}
\newcommand\cF{\mathcal{F}} 
\newcommand\cI{\mathcal{I}} 
\newcommand\G{\mathcal{G}}
\newcommand\cB{\mathcal{B}}
\newcommand\cT{\mathcal{T}}
\newcommand\cS{\mathcal{S}}
\newcommand\cP{\mathcal{P}}
\newcommand\K{\mathcal{K}}

\newcommand\GL{{ \rm  GL}}
\newcommand\Gl{{ \rm  GL}}
\newcommand\U{{ \rm  U}}
\def\SU{{\rm SU}}
\def\S{{\bf S}}
\newcommand\Gsp{{\rm Gsp}}
\newcommand\Lie{{\rm Lie}} 
\newcommand\Sl{{\rm  SL}}
\newcommand\SL{{\rm  SL}}
\newcommand\rO{{\rm  O}}
\newcommand\rU{{\rm  U}}
\newcommand\SO{{\rm  SO}}
\newcommand{\Sp}{\text{Sp}}
\newcommand\Ad{\text{Ad}}
\newcommand\AI{\text{AI}}
\newcommand\Sym{\text{Sym}}
\newcommand\Lef{{\rm Lef}}

\def\ringO{\mathcal{O}}
\def\idealP{\mathfrak{P}} 
 \def\g{\mathfrak{g}}
\def\k{\mathfrak{k}}
\def\z{\mathfrak{z}}
\def\s{\mathfrak{s}}
\def\c{\mathfrak{c}}
\def\b{\mathfrak{b}}
\def\t{\mathfrak{t}}
\def\q{\mathfrak{q}}
\def\l{\mathfrak{l}}
\def\gl{\mathfrak{gl}}
\def\sl{\mathfrak{sl}}
\def\u{\mathfrak{u}} 
\def\su{\mathfrak{su}}
\def\fp{\mathfrak{p}} 
\def\p{\mathfrak{p}}   
\def\r{\mathfrak{r}}
\def\a{\mathfrak{a}}
\def\n{\mathfrak{n}}
\def\fd{\mathfrak{d}}
\def\fR{\mathfrak{R}}
\def\fI{\mathfrak{I}}
\def\fJ{\mathfrak{J}}
\def\i{\mathfrak{i}}
\def\perm{\mathfrak{S}}
\newcommand\fg{\mathfrak g}
\newcommand\fk{\mathfrak k}
\newcommand\fgK{(\mathfrak{g},K_\infty^\circ)}
\newcommand\gK{ \mathfrak{g},K_\infty^\circ}

\newcommand\ul{\underline} 
\newcommand\tp{{  {\pi}_f}}
\newcommand\tv{ {\pi}_v}
\newcommand\ts{{ {\sigma}_f}}
\newcommand\pts{ {\sigma}^\prime_f}
\newcommand\usf{\ul{\sigma}_f}
\newcommand\pusf{\ul{\sigma}^\prime_f}
\newcommand\usvp{\ul{\sigma}^\prime_v}
\newcommand\usv{\ul{\sigma}_v}
\newcommand\io{{}^\iota}
\newcommand\uls{{\underline{\sigma}}}

\newcommand\Spec{\hbox{\rm Spec}} 
\newcommand\SGK{\mathcal{S}^G_{K_f}}
\newcommand\SMP{\mathcal{S}^{M_P}}
 \newcommand\SGn{\mathcal{S}^{G_n}}
 \newcommand\SGp{\mathcal{S}^{G_{n^\prime}}}
\newcommand\SMPK{\mathcal{S}^{M_P}_{K_f^{M_P}}}
\newcommand\SMQ{\mathcal{S}^{M_Q}}
\newcommand\SMp{\mathcal{S}^{M }_{K_f^M}}
\newcommand\SMq{\mathcal{S}^{M^\prime}_{K_f^{M^\prime}}}
\newcommand\uSMP{\ul{\mathcal{S}}^{M_P}}
\newcommand\SGnK{\mathcal{S}^{G_n}_{K_f}}
\newcommand\SG{\mathcal{S}^G}
\newcommand\SGKp{\mathcal{S}^G_{K^\prime_f}}
\newcommand\piKK{{ \pi_{K_f^\prime,K_f}}}
\newcommand\piKKpkt{\pi^{\pkt}_{K_f^\prime,K_f}}
\newcommand\BSC{ \bar{\mathcal{S}}^G_{K_f}}
\newcommand\PBSC{\partial\SGK}
\newcommand\pBSC{\partial\SG}
\newcommand\PPBSC{\partial_P\SGK}
\newcommand\PQBSC{\partial_Q\SGK}
\newcommand\ppBSC{\partial_P\mathcal{S}^G}
\newcommand\pqBSC{\partial_Q\mathcal{S}^G}
\newcommand\prBSC{\partial_R\mathcal{S}^G}
\newcommand \bs{\backslash} 
 \newcommand \tr{\hbox{\rm tr}}
 \newcommand\ord{\text{ord}}
\newcommand \Tr{\hbox{\rm Tr}}
\newcommand\HK{\mathcal{H}^G_{K_f}}
\newcommand\HKS{\mathcal{H}^G_{K_f,\place}}
\newcommand\HKv{\mathcal{H}^G_{K_v}}
\newcommand\HGS{\mathcal{H}^{G,\place}}
\newcommand\HKp{\mathcal{H}^G_{K_p}}
\newcommand\HKpo{\mathcal{H}^G_{K_p^0}}
\newcommand\ch{{\bf ch}}

\newcommand\M{\mathcal{M}}
\newcommand\Ml{\M_\lambda}
\newcommand\tMl{\tilde{\Ml}}
\newcommand\tM{\widetilde{\mathcal{M}}}
\newcommand\tMZ{\tM_\Z}
\newcommand\tsigma{\ul{\sigma}}
\newcommand \pkt{\bullet}
\newcommand\tH{\widetilde{\mathcal{H}}}
\newcommand\Mot{{\bf M}} 
\newcommand\eff{{\rm eff}}
\newcommand\Aql{A_{\q}(\lambda)}
\newcommand\wl{w\cdot\lambda}
\newcommand\wlp{w^\prime\cdot\lambda} 

\def\w{{\bf w}} 
\def\d{{\sf d}}
\def\e{{\bf e}} 
\def\x{{\tt x}}
\def\y{{\tt y}}
\def\v{{\sf v}}
\def\q{{\sf q}} 
\def\ff{{\bf f}}
\def\bk{{\bf k}}
 
\def\Ext{{\rm Ext}}
\def\Aut{{\rm Aut}}
\def\Hom{{\rm Hom}}
\def\Ind{{\rm Ind}}
\def\aInd{{}^{\rm a}{\rm Ind}}
\def\aIndPG{\aInd_{\pi_0(P(\R)) \times P(\A_f)}^{\pi_0(G(\R)) \times G(\A_f)}}
\def\aIndQG{\aInd_{\pi_0(Q(\R)) \times Q(\A_f)}^{\pi_0(G(\R)) \times G(\A_f)}}
\def\Gal{{\rm Gal}}
\def\End{{\rm End}} 
\def\cm{{\rm cm}} 
\newcommand\Coh{{\rm Coh}}  
\newcommand\Eis{{\rm Eis}}
\newcommand\Res{\mathrm{Res}}
\newcommand\place{\mathsf{S}}
\newcommand\emb{\mathcal{I}} 
\newcommand\LB{\mathcal{L}}  
\def\Hod{{\mathcal{H}od}}
\def\Crit{{\rm Crit}}
  
 \newcommand\ip{\pi_f\circ \iota} 
\newcommand\Wp{W_{\pi_\infty\times \ip}}
\newcommand\Wpc{W^{\text{cusp}}_{\pi_\infty\times \pi_f\circ\iota}}
\newcommand\Lcusp{  L^2_{\text{cusp}}(G(\Q)\bs G(\A_f)/K_f)}
\newcommand\MiC{\tM_{\iota\circ\lambda,\C} }
\newcommand\miC{\M_{\iota\circ\lambda,\C} }
\newcommand\tpl{{}^\iota}
\newcommand\Id{\rm Id}
\newcommand\Lr{L^{\text{rat}}}
\newcommand \iso{ \buildrel \sim \over\longrightarrow} 
\newcommand\us{\ul\sigma}
\newcommand\qvs{q_v^{-z}}
\newcommand \into{\hookrightarrow}
\newcommand\ppfeil[1]{\buildrel #1\over \longrightarrow}
\newcommand\eb{{}^\iota}
    
\def\bfpi{\mathbf{\Pi}}
\def\bfdelta{\mathbf{\Delta}}

\def\sI{\mathscr{I}}
\def\sU{\mathscr{U}}
\def\sJ{\mathscr{J}}


\section{Introduction}
In the arithmetic theory of automorphic forms, it has been a long standing problem to show the existence of cuspidal cohomology classes on a locally symmetric space with values in a given local system. Even though there has been a lot of progress over the last half a century, the problem is very far from being resolved. The main goal of this paper is to prove the nonvanishing of cuspidal cohomology for $\GL(n)$ over a number field which is a Galois extension over its maximal totally real subfield. This result seems to be new even for $n=2$.

\medskip

To state the main result more precisely, suppose $F$ is a number field and $F_0$ its maximal totally real subfield of $F$. 
Assume that $F/F_0$ is a Galois extension. For example, this is automatically satisfied if $F/\Q$ is itself a Galois extension. 
Let $G_0 = \GL_n/F$, and $G = \Res_{F/\Q}(G_0)$ the group obtained from $G_0$ by the Weil restriction of scalars from $F$ to $\Q$. As will become apparent in the proofs, the technical 
artifice of considering $\GL_n/F$ as a $\Q$-group is useful. 
For an open compact subgroup $K_f$ of $G(\A_{f})$, the locally symmetric space for $G$ with level $K_f$ is denoted $S^G_{K_f}$. 
Let $T_0$ be a maximal torus in $G_0$, and $T = \Res_{F/\Q}(T_0)$. Let $E/\Q$ be a finite Galois extension that takes a copy of $F$; then $E$ splits $G$. 
Suppose $\lambda \in X^+(\Res_{F/\Q}(T_0) \times E)$ is a dominant integral weight, $\M_{\lambda, E}$ the absolutely-irreducible 
finite-dimensional representation of $G \times E$ of highest weight $\lambda$, and $\tM_{\lambda, E}$ the corresponding sheaf of $E$-vector spaces on $S^G_{K_f}$. 
Consider the sheaf-theoretically defined cohomology group $H^\bullet(\SGK, \tM_{\lambda, E}).$  Passing to a transcendental context via an embedding 
$\iota : E \to \C,$ one can study $H^\bullet(\SGK, \tM_{{}^\iota\lambda, \C})$ as the relative Lie algebra cohomology 
of the space of automorphic forms on $\SGK$ with values in $\tM_{{}^\iota\lambda, \C}.$ Within this space lives cuspidal cohomology $H_{\rm cusp}^\bullet(\SGK, \tM_{{}^\iota\lambda, \C})$ which 
was originally considered by Eichler and Shimura. 
The details of cuspidal cohomology are reviewed in Sect.\,\ref{sec:prelims-gln}. 
One knows from Clozel's purity lemma (\cite[Lemme de puret\'e 4.9]{clozel})
 that if cuspidal cohomology is nonzero then ${}^\iota\lambda$ satisfies a purity condition. Furthermore, since cuspidal cohomology has a rational structure (\cite[Th\'eor\`eme 3.19]{clozel})
$\lambda$ itself is pure, in the sense that ${}^\iota\lambda$ satisfies the same purity condition for every embedding $\iota : E \to \C;$ we will say that $\lambda$ {\it strongly-pure.} 
In Sect.\,\ref{sec:pure} we study strongly-pure weights; see Prop.\,\ref{prop:strong-pure-E} for the various ways to characterize them. 
Strong-purity is necessary for $\lambda$ to support cuspidal cohomology; the basic problem is whether it is also sufficient; see Conj.\,\ref{con:main} for a precise statement. 

In Prop.\,\ref{prop:strong-pure-weights-base-change} we prove that 
strong-purity imposes the internal structural restriction on $\lambda$ that it is the base-change from a (strongly-)pure weight of 
$X^+(\Res_{F_0/\Q}(T_0) \times E)$. 
We are able to leverage this internal restriction on $\lambda$ with the foundational work of Borel, Labesse, 
and Schwermer \cite{borel-labesse-schwermer} reviewed in Sect.\,\ref{sec:BLS}, while harnessing inputs from Vogan and Zuckerman \cite{vogan-zuckerman} 
on canonical $K$-types supporting cohomology that lets us compute an archimedean Lefschetz number (see Prop.\,\ref{prop:slnc-lefschetz}) whose nonvanishing  
proves the main theorem of this paper: 

\begin{thm}[Theorem \ref{thm:main-gln}]
Let $F$ be a number field which is Galois over its maximal totally real subfield. Let $G = \Res_{F/\Q}(\GL_n/F)$; rest of the notations as above. 
Suppose $\lambda \in X^+(\Res_{F/\Q}(T_0) \times E)$ is a strongly-pure weight, then we have nonvanishing of cuspidal cohomology for some level structure, i.e., for some deep enough open-compact subgroup $K_f$ of $G(\A_f)$ we have 
$$H_{\rm cusp}^\bullet(\SGK, \tM_{{}^\iota\lambda, \C}) \neq 0$$ 
for every embedding $\iota : E \to \C$. 
\end{thm}

There is an interesting technical subproblem to overcome 
in applying \cite{borel-labesse-schwermer}: we need to relate the nonvanishing of cuspidal cohomology 
of $\GL(n)$ with which we are essentially interested to that of $\SL(n)$. This is addressed in Sect.\,\ref{sec:gln-sln} where the results of Labesse and Schwermer 
\cite{labesse-schwermer-JNT} play a critical role. 

\medskip

The above theorem has been known in various special cases: for $F$ totally real and $n=2$ it boils down to classical results on the existence of Hilbert modular forms of prescribed weight. 
More generally, the reader is referred to 
Labesse--Schwermer \cite{labesse-schwermer} when the field $F$ is filtered over a totally real field 
by cyclic or cubic extensions, $n \in \{2, 3\},$ and for $\lambda = 0$, i.e., for trivial coefficients; 
Clozel \cite{clozel-duke} for general $F$, $n$ even, and $\lambda = 0$; 
Borel--Labesse--Schwermer \cite{borel-labesse-schwermer} for general $n$, $\lambda = 0$, and $F$ filtered over a totally real field by cyclic extensions; 
Rajan \cite{rajan} for $\SL_1(D)$ for a quaternionic division algebra $D$ with some additional restrictions on $F$ and $\lambda$; 
Grobner \cite{grobner} for $\GL_2(D)$ for a quaternionic division algebra $D$ over $\Q$ and a self-dual $\lambda$; 
Bhagwat--Raghuram \cite{bhagwat-raghuram-BIMS} for $F$ totally real 
or a totally imaginary quadratic extension over a totally real field and $\GL(n)$ for general $n$ using Arthur classification for classical groups \cite{arthur-book} and 
Clozel's limit multiplicity theorems \cite{clozel-inventiones}. 
Some other related works are Rolhfs--Speh \cite{rohlfs-speh-ENS} who introduced the technique of computing certain automorphic Lefschetz numbers; see also 
Barbasch--Speh \cite{barbasch-speh}; Rolhfs--Speh \cite{rohlfs-speh-Nach} and Li--Schwermer \cite{li-schwermer} for similar questions for other reductive groups; 
Boxer--Calegari--Gee \cite{boxer-calegari-gee} where $F = \Q$ and the level structure $K_f$ is the full maximal open compact in $G(\A_f)$ rendering the situation to be arithmetically delicate. The interested reader may survey the literature on cuspidal cohomology using the sampling of papers mentioned in this paragraph. 

\medskip

Nonvanishing of cuspidal cohomology is fundamentally about the geometry of a locally symmetric space which has important applications in the arithmetic theory of the special values of $L$-functions 
attached to cohomological cuspidal automorphic forms. See, for example, Harder--Raghuram \cite[Chap.\,5]{harder-raghuram-book}. 
Indeed, this was also our original motivation to study this problem. Generalizations of the main theorem and such arithmetic applications will be taken up in some future work.

\medskip
\section{Preliminaries on the cuspidal cohomology of $\GL(n)$}
\label{sec:prelims-gln}

This section is to set up our notations. We will be brief here and refer the reader to \cite{harder-raghuram-book} and \cite{raghuram} for more details. 

\subsection{The base field}
\label{sec:base-field}
Let $F$ be a number field by which we mean a finite extension of $\Q$, and although not strictly necessary, it is convenient to take $F$ to be contained inside the field $\C$ of complex numbers.  
Let $\Sigma_F = \Hom(F,\C)$ be the set of all embeddings of $F$ into $\C.$ 
The completion of $F$ with respect to a place $v$ is denoted $F_v$. Suppose 
$S_r$ and $S_c$ denote the sets of real and complex places of $F$ respectively, then $S_\infty = S_r \cup S_c$ denotes the set of all archimedean places. 
An element $v \in S_r$ corresponds to a unique embedding 
$\tau_v \in \Hom(F,\R)$ that extends to an isomorphism $F_v \simeq \R$, 
whereas a $v \in S_c$ corresponds to a pair $\{\tau_v, \overline{\tau}_v \} \subset \Sigma_F$ of conjugate embeddings in which a non-canonical 
choice of $\tau_v$ is fixed that extends to an isomorphism $F_v \simeq \C.$ (See \ref{sec:action-of-c} where making this non-canonical choice plays a role.)
Let $F_{\infty}:=F\otimes_{\Q} \R$. For a finite set $S=S_f\cup S_{\infty}$ of places containing all archimedean places, let $F_S=\prod_{v\in S} F_v$. 

\smallskip

Let $F_0$ be the unique maximal totally real subfield of $F$. Even if $F$ is totally imaginary, it may or may not contain a totally imaginary quadratic extension of $F_0$; if it does, then let $F_1$ denote such a subfield---it is the unique maximal CM-subfield of $F$; and if $F$ does not contain a CM-subfield then for homogeneity of notations, put $F_1 = F_0.$

\subsection{Locally symmetric spaces}
\label{sec:loc-symm-space}

Let $G_0= \GL_n/F$ and $G = \Res_{F/\Q}(G_0)$ the group obtained from $G_0$ by the Weil restriction of scalars from $F$ to $\Q$. 
Let $B_0$ be the Borel subgroup of upper triangular matrices in $G_0$, $T_0$ the maximal torus in $B_0$, and $Z_0$ the center of $G_0.$ The corresponding groups obtained 
via Weil restriction of scalars from $F$ to $\Q$ are denoted $B$, $T$, and $Z$, respectively. Let $S$ denote the maximal $\Q$-split torus in $Z$; note that $S \cong {\mathbb G}_m/\Q.$ 
For the group of $\R$-points one has $G(\R)\cong \prod_{v \in S_r}\GL_n(\R)\times \prod_{v \in S_c}\GL_n(\C)$, within which $Z(\R)$ sits as $\prod_{v \in S_r} \R^{\times} I_n \times \prod_{v \in S_c}\C^{\times}I_n,$ where $I_n$ denotes the $n \times n$ identity matrix. Furthermore, $S(\R) = \R^\times$ sits diagonally in $Z(\R)$. 
Let $C_{\infty} \cong \prod_{v \in S_r} \rO(n)\times \prod_{v \in S_c} \rU(n)$ be the usual maximal compact subgroup of $G(\R)$. 
Let $K_{\infty}=C_{\infty}S(\R)$  and $K_{\infty}^{\circ}$ its group of connected component of the identity. 
Define the symmetric space associated to $G$ as $S^G := G(\R)/K_{\infty}^{\circ}$. 
For an open compact subgroup $K_f$ of $G(\A_{f})$ define the adelic symmetric space as $G(\A)/K_{\infty}^{\circ}K_f\cong S^G\times G(\A_{f})/K_f.$ Since $G(\Q)$ is a 
discrete subgroup of $G(\A)$, it acts properly discontinuously on $G(\A)/K_{\infty}^{\circ}K_f$, and going modulo this action we get the locally symmetric space 
$S^G_{K_f}=G(\Q)\backslash G(\A)/K_{\infty}^{\circ}K_f.$
for $G$ with level $K_f$. Let $\pi: S^G \rightarrow S^G_{K_f}$ be the canonical map of quotienting by $G(\Q)$. If necessary, by passing to a subgroup of finite index of $K_f$, one has a finite disjoint union 
$S^G_{K_f}=\bigsqcup \Gamma_i\backslash G(\R)/K_{\infty}^{\circ}$ with each $\Gamma_i$ being torsion-free. In this article we will always assume that $K_f$ satisfies this property.

\medskip
\subsection{Sheaves on locally symmetric spaces}

\medskip
\subsubsection{\bf The field of coefficients $E$} 
Let $E/\Q$ be a `large enough' finite Galois extension that takes a copy of $F.$ This field $E$ will be the field of coefficients. Whereas our ultimate object of interest is cuspidal cohomology---a transcendental object---for which one could drop finiteness and take $E = \C$. However, keeping a larger context in mind, wherein one is interested in arithmetic applications 
of cuspidal cohomology, it is best to set things up over such a general field $E$ of coefficients. An embedding $\iota : E \to \C$ gives a bijection 
$\iota_* : \Hom(F,E) \to \Hom(F,\C)$ given by composition: $\iota_*\tau = \iota \circ \tau.$

\medskip
\subsubsection{\bf Characters of the torus $T$}
\label{sec:characters-of-torus}
For $E$ as above, let 
$X^*(T \times E) := \Hom_{E-{\rm alg}}(T \times E, \mathbb{G}_m),$
where $ \Hom_{E-{\rm alg}}$ is to mean homomorphisms of $E$-algebraic groups. There is a natural action of ${\rm Gal}(E/\Q)$ on 
$X^*(T \times E)$. Since $T = \Res_{F/\Q}(T_0)$, and $T_0$ is split over $F$, one has  
$$
X^*(T \times E) \ = \ \bigoplus_{\tau : F \to E} X^*(T_0 \times_{F,\tau} E)  \ = \ \bigoplus_{\tau : F \to E} X^*(T_0).  
$$
Let $X^*_\Q(T \times E) = X^*(T \times E) \otimes \Q.$ 
The weights are parametrized as in \cite{harder-raghuram-book}: $\lambda \in X^*_\Q(T \times E)$ will be written as 
$\lambda = (\lambda^\tau)_{\tau : F \to E}$ with 
$$
\lambda^\tau \ = \ \sum_{i=1}^{n-1} (a^\tau_i-1)  \bfgreek{gamma}_i \ + \ d^\tau \cdot \bfgreek{delta}_n 
 \ = \ (b^\tau_1, b^\tau_2, \dots, b^\tau_n), 
$$
where, $\bfgreek{gamma}_i$ is the $i$-th fundamental weight for $\SL_n$ extended to $\GL_n$ by making it trivial on the center, and $\bfgreek{delta}_n$ is the determinant character of $\GL_n.$ If $r_{\lambda} := (nd - \sum_{i=1}^{n-1} i (a_i-1))/n,$ 
then $b_1 =  a_1 + a_2 + \dots + a_{n-1} - (n-1) + r_{\lambda}, \ 
b_2  =  a_2 + \dots + a_{n-1} - (n-2) + r_{\lambda}, \dots, b_{n-1} =  a_{n-1} - 1 + r_{\lambda}, \ 
b_n  =  r_{\lambda},$ and conversely, 
$a_i  - 1 =  b_i - b_{i+1}, \ d  =  (b_1+\dots+b_n)/n.$

\smallskip

A weight $\lambda = \sum_{i=1}^{n-1} (a_i-1) \bfgreek{gamma}_i+ d \cdot \bfgreek{delta}_n  = (b_1,\dots,b_n)  \in X^*_\Q(T_0)$ is integral if and only if 
$$
\lambda \in X^*(T_0) \ \Longleftrightarrow \ 
b_i \in \Z, \ \forall i 
 \ \Longleftrightarrow \ 
\left\{\begin{array}{l} 
a_i \in \Z, \quad 1 \leq i \leq n-1, \\
nd \in \Z, \\
nd \equiv \sum_{i=1}^{n-1} i (a_i-1) \pmod{n}.
\end{array}\right.
$$
A weight $\lambda = (\lambda^\tau)_{\tau : F \to E} \in X^*_\Q(T \times E)$ is integral if and only if each $\lambda^\tau$ is integral.
Next, an integral weight $\lambda \in X^*(T_0)$ is dominant, for the choice of the Borel subgroup being $B_0$,  if and only if 
$$
b_1 \geq b_2 \geq \dots \geq b_n
 \ \Longleftrightarrow \ a_i \geq 1 \  
 \mbox{for $1 \leq i \leq n-1.$ \ (No condition on $d$.)}
$$
A weight $\lambda = (\lambda^\tau)_{\tau : F \to E} \in X^*_\Q(T \times E)$ is dominant-integral if and only if each $\lambda^\tau$ is dominant-integral. Let 
$X^+(T \times E)$ stand for the set of all dominant-integral weights.

\smallskip
An embedding $\iota : E \to \C$ induces a bijection $X^*(T\times E) \to X^*(T\times \C)$
defined as 
$$
\lambda=(\lambda^{\tau})_{\tau: F \to E} 
 \ \ \mapsto \ \
 {^{\iota}}\lambda=(^{\iota}\lambda^{\eta})_{\eta : F \to \C},  
 \quad ^{\iota}\lambda^{\eta}=\lambda^{\iota^{-1}\circ\eta}.
$$

\medskip
\subsubsection{\bf The sheaf $\tM_{\lambda, E}$}
\label{sec:the-sheaf}
For $\lambda \in X^+(T \times E)$, put 
 $\M_{\lambda, E} \ = \ \bigotimes_{\tau : F \to E} \M_{\lambda^\tau},$
 where $\M_{\lambda^\tau}/E$ is the absolutely-irreducible finite-dimensional representation of 
 $G_0 \times_\tau E = \GL_n/F \times_\tau E$ with highest weight  $\lambda^\tau.$ Denote this representation as 
 $(\rho_{\lambda^\tau}, \M_{\lambda^\tau})$. The group $G(\Q) = \GL_n(F)$ acts on $\M_{\lambda, E}$ diagonally, i.e., 
$a \in G(\Q)$ acts on a pure tensor
$\otimes_\tau m_\tau$ via: 
$a \cdot (\otimes_\tau m_\tau) \ = \ \otimes_\tau \rho_{\lambda^\tau}(\tau(a))(m_\tau).$
This representation gives a sheaf $\tM_{\lambda, E}$ of $E$-vector spaces on $\SGK$ (see \cite{harder-raghuram-book} and \cite{raghuram}). 
The sheaf $\tM_{\lambda, E}$ is nonzero only if 
the central character of $\rho_\lambda$ has the type of an algebraic Hecke character of $F$ (see \cite[1.1.3]{harder-inventiones} and \cite[Lem.\,17]{raghuram-hecke}), 
which means: 
\begin{itemize}
    \item if $F$ has a real place, then there exists $w\in \Z$ such that $d^{\eta}=w$ for all $\eta\in \Sigma_{F}$; 
    \item if $F$ is totally imaginary then there exists $w\in \Z$ such that $d^{\iota^{-1}\circ\eta}+d^{\iota^{-1}\circ \overline{\eta}}=w$ for all $\tau \in \Sigma_{F}$ and $\iota: E\rightarrow \C$.
\end{itemize}
A dominant integral weight is called algebraic if it satisfies the above condition;  the set of algebraic dominant integral weights will be denoted $X^+_{\rm alg}(T\times E)$. 


\medskip
\subsection{Cohomology of $\tM_{\lambda, E}$} 
\label{sec:long-e-seq}

For $\lambda \in X^+_{\rm alg}(T \times E)$, a basic object of study is the sheaf-cohomology group $H^\bullet(\SGK, \tM_{\lambda,E})$. One of the main tools is a long exact sequence coming from the Borel--Serre compactification. 
Let $\BSC = \SGK \cup \partial\SGK$ be the Borel--Serre compactification of $\SGK$. The sheaf $\tM_{\lambda,E}$ extends naturally to $\partial\SGK$ and so also to $\BSC$, and 
the pair $(\BSC, \partial\SGK)$ induces the following long-exact sequence in cohomology: 
\begin{multline}
\label{eqn:lon-ex-seq}
\cdots  \longrightarrow H^i_c(\SGK, \tM_{\lambda, E}) 
\stackrel{\mathfrak{i}^\bullet}{\longrightarrow}   H^i(\SGK, \tM_{\lambda,E}) 
\stackrel{\mathfrak{r}^\bullet}{\longrightarrow } H^i(\partial \SGK, \tM_{\lambda,E}) \stackrel{\fd^\bullet}{\longrightarrow} \\
\stackrel{\fd^\bullet}{\longrightarrow} H^{i+1}_c(\SGK, \tM_{\lambda,E}) \longrightarrow \cdots
\end{multline}
This is an exact sequence of Hecke modules, for the action of the Hecke algebra $\HK = C^\infty_c(G(\A_f)/ \! \!/K_f)$ of all locally constant and compactly supported bi-$K_f$-invariant $\Q$-valued functions on $G(\A_f);$ the Haar measure on $G(\A_f)$ to be the product of local Haar measures, and for every prime $p$, the local measure is normalized so that 
$\mathrm{vol}(G(\Z_p)) = 1;$ then $\HK$ is a $\Q$-algebra under convolution of functions. 
The image of cohomology with compact supports inside the full cohomology is called {\it inner} or {\it interior} cohomology and is denoted 
$H^{\bullet}_{\, !} := {\rm Image}(\mathfrak{i}^\bullet) = {\rm Im}(H^{\bullet}_c \to H^{\bullet}).$ 
Inner cohomology is a semi-simple module for the Hecke-algebra. 
If $E/\Q$ is sufficiently large, then there is an isotypical decomposition:  
\begin{align}
\label{deco}  
H^\bullet_{\,!}(\SGK, \M_{\lambda,E}) \ =\ 
\bigoplus_{\pi_f \in {\rm Coh}_!(G,K_f,\lambda)} H^\bullet_{\,!}(\SGK, \M_{\lambda,E})(\pi_f), 
\end{align}
where $\pi_f$ is an isomorphism type of an absolutely  irreducible $\HK$-module, i.e., there is an $E$-vector space $V_{\pi_f}$ with an absolutely irreducible action $\pi_f$ of $\HK$.  
The set ${\mathrm{Coh}}_!(G, K_f, \lambda)$ of isomorphism classes which occur with strictly positive multiplicity in \eqref{deco} is called the inner spectrum of $G$ with $\lambda$-coefficients and level structure $K_f.$ Taking the union over all $K_f$, the inner spectrum of $G$ with $\lambda$-coefficients is defined to be: 
${\mathrm{Coh}}_!(G, \lambda) \ = \ \bigcup_{K_f} {\mathrm{Coh}}_!(G, K_f, \lambda). $

\medskip
\subsection{Cuspidal cohomology} 
\label{sec:cuspidal-coh}
To study these cohomology groups one passes to a transcendental situation using an embedding $E \hookrightarrow \C$, and then use the theory of automorphic forms on $G$.
To simplify notations, for the moment, take $E = \C$ and $\lambda \in X^+_{\rm alg}(T \times \C).$ 
Denote $\fg_\infty$ (resp., $\fk_\infty$) the Lie algebra of $G(\R)$ (resp., of $K_\infty = C_\infty S(\R).$)
The cohomology $H^\bullet(\SGK,\tM_{\lambda, \C})$ is the cohomology of the de~Rham complex $\Omega^\bullet(\SGK, \tM_{\lambda, \C}),$ which 
is isomorphic to the relative Lie algebra complex: 
$$
\Hom_{K_\infty^\circ} (\Lambda^\bullet(\fg_\infty/\fk_\infty), \, 
\cC^\infty(G(\Q)\backslash G(\A)/K_f, \omega_{\lambda}^{-1}|_{S(\R)^\circ}) \otimes  \M_{\lambda, \C}),
$$
 where $\cC^\infty(G(\Q)\backslash G(\A)/K_f, \omega_{\lambda}^{-1}|_{S(\R)^\circ})$ consists of all smooth functions 
 $\phi : G(\A) \to \C$ such that $\phi(a \, \ul g \, \ul k_f \, s_\infty) = \omega_{\lambda}^{-1}(s_\infty) \phi(\ul g),$ 
 for all $a \in G(\Q)$, $\ul g \in G(\A)$, $\ul k_f \in K_f$ and $s_\infty \in S(\R)^\circ.$ 
 Abbreviating $\omega_{\lambda}^{-1}|_{S(\R)^\circ}$ as $\omega_\infty^{-1},$ if $t \in \R_{>0} \cong S(\R)^\circ$ then 
 $ \omega_{\lambda}(t) \ = \ t^{N \sum_{\tau : F \to \C} d^\tau} \ = \ t^{\sum_\tau \sum_i b_i^\tau}.$
 The identification of the complexes gives an identification of $H^\bullet(\SGK, \tM_{\lambda,\C})$ 
 with the relative Lie algebra cohomology of the space of smooth automorphic forms twisted by the coefficient system: 
 $$
H^\bullet(\SGK, \tM_{\lambda,\C}) \ \cong \ 
H^\bullet(\fg_\infty, \fk_\infty; \, \cC^\infty(G(\Q)\backslash G(\A)/K_f, \omega_{\lambda}^{-1}|_{S(\R)^\circ}) \otimes  \M_{\lambda, \C}).
$$
The inclusion $\cC_{\rm cusp}^\infty(G(\Q)\backslash G(\A)/K_f, \omega_\infty^{-1}) \subset   
  \cC^\infty(G(\Q)\backslash G(\A)/K_f, \omega_\infty^{-1}),$ 
of the space of smooth cusp forms,  
induces an inclusion in relative Lie algebra cohomology (due to Borel \cite{borel-duke}), and cuspidal cohomology is defined as: 
\begin{equation}
\label{eqn:def-cusp-coh}
 H_{\rm cusp}^\bullet(\SGK, \tM_{\lambda, \C}) \ := \ 
H^\bullet \left(\fg_\infty, \fk_\infty;  \, 
\cC^\infty_{\rm cusp}(G(\Q)\backslash G(\A)/K_f, \omega_\infty^{-1})  \otimes \M_{\lambda, \C} \right).
\end{equation}
One may, whenever convenient, pass to the limit over all level structures $K_f$ and define:
\begin{equation}
\label{eqn:def-cusp-coh-no-Kf}
H_{\rm cusp}^\bullet(G, \M_{\lambda, \C}) \ := \ 
H^\bullet \left(\fg_\infty, \fk_\infty;  \, 
\cC^\infty_{\rm cusp}(G(\Q)\backslash G(\A), \omega_\infty^{-1})  \otimes \M_{\lambda, \C} \right).
\end{equation}

Furthermore, one knows that the image of cuspidal cohomology under the restriction map $\mathfrak{r}^\bullet$ in \eqref{eqn:lon-ex-seq} is trivial, and hence 
$H_{\rm cusp}^\bullet(\SGK, \M_{\lambda, \C}) \subset H_{!}^\bullet(\SGK, \M_{\lambda, \C}).$ 
Define $\Coh_{\rm cusp}(G,\lambda,K_f)$ as the set of all $\pi_f \in \Coh_!(G, \lambda,K_f)$ 
which contribute to cuspidal cohomology.  The decomposition of cuspforms into cuspidal automorphic representations, 
gives the following fundamental decomposition for cuspidal cohomology: 
\begin{equation}
\label{eqn:cuspidal-coh-spectrum}
 H_{\rm cusp}^\bullet(\SGK, \tM_{\lambda, \C}) \ := \ 
 \bigoplus_{\pi \in \Coh_{\rm cusp}(G,\lambda,K_f)}  
 H^\bullet(\fg_\infty, \fk_\infty; \pi_\infty \otimes \M_{\lambda, \C}) \otimes \pi_f.
\end{equation}
To clarify a slight abuse of notation: if a cuspidal automorphic representation $\pi$ contributes to the above decomposition, then its representation at infinity is $\pi_\infty$, however 
$\pi_f$ denotes the $K_f$-invariants of the finite part of $\pi$. The level structure 
$K_f$ will be clear from context, hence whether $\pi_f$ denotes the finite-part or its $K_f$-invariants will be clear from context.  
Define ${\mathrm{Coh}}_{\rm cusp}(G, \lambda) \ = \ \bigcup_{K_f} {\mathrm{Coh}}_{\rm cusp}(G, K_f, \lambda).$

\section{Strongly pure weights}
\label{sec:pure} 
  
\medskip
\subsection{\bf Strongly-pure weights over $\C$}
\label{sec:strong-pure-C} 
Consider a weight $\lambda = (\lambda^\eta)_{\eta:F \to \C} \in X^+_{\rm alg}(T \times \C)$, where 
$\lambda^\eta \ = \ \sum_{i=1}^{n-1} (a^\eta_i-1)  \bfgreek{gamma}_i \ + \ d^\eta \cdot \bfgreek{delta}_n =  (b^\eta_1, \dots, b^\eta_n)$. 
If $\lambda$ supports cuspidal cohomology, i.e., if $H_{\rm cusp}^\bullet(\SGK, \tM_{\lambda, \C}) \neq 0$, then 
$\lambda$ satisfies the purity condition: 
\begin{equation}
\label{eqn:purity-def}
a_i^\eta = a_{n-i}^{\bar\eta} \ \mbox{for all $\eta : F \to \C$}\ \iff \ 
 \mbox{$\exists \,\w$ such that $b^\eta_i + b^{\bar\eta}_{n-i+1} = \w$ for all $\eta$ and $i$,}
\end{equation}
which follows from the purity lemma \cite[Lem.\,4.9]{clozel}. The integer $\w$ is called the {\it purity weight} of $\lambda.$ Taking the sum of the  
equation $b^\eta_i + b^{\bar\eta}_{n-i+1} = \w$ over all $i$, one deduces that $d^\eta = \w/2$ for all $\eta.$ 
The weight $\lambda$ is said to be {\it pure} if it 
satisfies \eqref{eqn:purity-def}, and denote by $X^+_{0}(T \times \C)$ the set of all such pure weights.

\medskip

Next, recall a theorem of Clozel which says that cuspidal cohomology for $\GL_n/F$ admits a rational structure 
\cite[Thm.\,3.19]{clozel}, from which it follows that any $\varsigma \in \Aut(\C)$ stabilizes cuspidal cohomology, i.e., ${}^\varsigma\lambda$ also 
satisfies the above purity condition, where if $\lambda = (\lambda^\eta)_{\eta: F \to \C}$ and $\varsigma \in \Aut(\C)$, then 
${}^\varsigma\lambda$ is the weight $({}^\varsigma\lambda^\eta)_{\eta: F \to \C}$ with ${}^\varsigma\lambda^\eta = \lambda^{\varsigma^{-1}\circ \eta}$. 
A pure weight $\lambda$ will be called {\it strongly-pure}, if ${}^\varsigma\lambda$ is pure with purity-weight $\w$ for every $\varsigma \in \Aut(\C);$ denote by 
$X^+_{00}(T \times \C)$ the set of all such strongly-pure weights. 
For $\lambda \in X^+_{00}(T \times \C),$ note that 
$$
b^{\varsigma^{-1} \circ \eta}_j + b^{\varsigma^{-1} \circ \overline{\eta}}_{n-j+1} = \w, \ \mbox{for all \ $1 \leq j \leq n, \ \eta : F \to \C, \ \varsigma \in \Aut(\C)$}. 
$$ 
We have the following inclusions inside the character group of $T \times \C,$ which are all, in general, strict inclusions: 
 $$
 X^+_{00}(T \times \C) \ \subset \  X^+_{0}(T \times \C) \ \subset \ X^+_{\rm alg}(T \times \C) \ \subset \ X^+(T \times \C) \ \subset \ X^*(T \times \C).
 $$

\medskip
\subsection{\bf Strongly-pure weights over $E$}
\label{sec:strongly-pure-E} 
 
The set of strongly-pure weights may be defined at an arithmetic level. Recall the standing assumption on $E$ that it is a finite Galois extension of $\Q$ that 
takes a copy of $F$; in particular, any embedding $\iota: E \to \C$ factors as $\iota: E \to \bar\Q \subset \C.$ See \cite[Prop.\,2.4]{raghuram} for the following proposition:

\begin{prop}
\label{prop:strong-pure-E}
Let $\lambda \in X^+_{\rm alg}(T \times E)$ be an algebraic dominant integral weight. Suppose 
$\lambda = (\lambda^\tau)_{\tau : F \to E}$ with $\lambda^\tau = (b^\tau_1 \geq \cdots \geq b^\tau_n)$. Then, the following are equivalent: 

\smallskip
\begin{enumerate}
\item[(i)] There exists $\iota : E \to \C$ such that ${}^\iota\lambda \in X^+_{00}(T \times \C)$, i.e., for every $\gamma \in \Gal(\bar\Q/\Q)$ we have 
${}^{\gamma \circ \iota}\lambda \in X^+_{0}(T \times \C)$ with the same purity weight: 
\begin{multline*}
``\exists \, \iota : E \to \C, \ \exists \, \w \in \Z \ \ \mbox{such that} \ \ 
b_j^{\iota^{-1}\circ \gamma^{-1} \circ \eta} + b_{n-j+1}^{\iota^{-1}\circ \gamma^{-1} \circ \bar\eta} \ = \ \w, \\ 
\forall \gamma \in \Gal(\bar\Q/\Q), \ \forall \eta : F \to \C, \ 1 \leq  j \leq n."
\end{multline*}

\smallskip
\item[(ii)] For every $\iota : E \to \C$, ${}^\iota\lambda \in X^+_{00}(T \times \C)$, i.e., for every $\gamma \in \Gal(\bar\Q/\Q)$ we have 
${}^{\gamma \circ \iota}\lambda \in X^+_{00}(T \times \C)$ with the same purity weight: 
\begin{multline*}
``\exists\, \w \in \Z \ \ \mbox{such that} \ \ 
b_j^{\iota^{-1}\circ \gamma^{-1} \circ \eta} + b_{n-j+1}^{\iota^{-1}\circ \gamma^{-1} \circ \bar\eta} \ = \ \w, \\  
\forall \, \iota : E \to \C, \ \forall\, \gamma \in \Gal(\bar\Q/\Q), \ \forall\, \eta : F \to \C, \ 1 \leq  j \leq n."
\end{multline*}

\smallskip
\item[(iii)] For every $\iota : E \to \C$, ${}^\iota\lambda \in X^+_{0}(T \times \C)$ with the same purity weight: 
$$
``\exists\, \w \in \Z \ \ \mbox{such that} \ \ 
b_j^{\iota^{-1}\circ \eta} + b_{n-j+1}^{\iota^{-1}\circ \bar\eta} \ = \ \w, \ 
\forall \, \iota : E \to \C,  \ \forall\, \eta : F \to \C, \ 1 \leq  j \leq n."
$$
\end{enumerate}
\end{prop}

\medskip 
The set of strongly-pure weights over $E$, denoted $X^+_{00}(T \times E)$ consists of the algebraic dominant integral weights 
$\lambda \in X^*(T \times E)$ that satisfy any one, and hence all, of the
conditions in the above proposition. It is most convenient to work with the characterization in $(iii)$.
There are the following inclusions within the character group of $T \times E$, which are all, in general, strict inclusions: 
 $$
 X^+_{00}(T \times E) \ \subset \ X^+_{\rm alg}(T \times E) \ \subset \ X^+(T \times E) \ \subset \ X^*(T \times E).
 $$

\medskip
\subsection{\bf Inner structure of strongly-pure weights}
\label{sec:base-change-of-weights}

The crucial fact about a strongly-pure weight is the following proposition that any such weight over a general number field $F$ is the base-change from a (strongly-)pure weight 
over $F_1$. Recall that $F_1$ is the maximal CM or maximal totally real subfield of $F$. This was observed in \cite{raghuram} where the base field was a totally imaginary number 
field. The statement and its proof go through just the same for a general number field. For the sake of completeness we state and prove it as the following proposition. 

\begin{prop}
\label{prop:strong-pure-weights-base-change}
Suppose $\lambda \in X^+_{00}(\Res_{F/\Q}(T_0) \times E)$ is a strongly-pure weight. Then there exists $\kappa \in X^+_{00}(\Res_{F_1/\Q}(T_0) \times E)$ such that 
$\lambda$ is the base-change of $\kappa$ from $F_1$ to $F$ in the sense that for any $\tau : F \to E$, $\lambda^\tau = \kappa^{\tau|_{F_1}}.$ 
\end{prop}

For brevity, the conclusion will be denoted as $\lambda = {\rm BC}_{F/F_1}(\kappa).$

\begin{proof}
It suffices to prove the proposition over $\C$, i.e., if $'\lambda \in X^+_{00}(\Res_{F/\Q}(T_0) \times \C)$ then it suffices to show the existence 
$'\kappa \in X^+_{00}(\Res_{F_1/\Q}(T_0) \times \C)$ such that $'\lambda = {\rm BC}_{F/F_1}('\kappa);$ because, then given the $\lambda$ in the proposition,
take an embedding $\iota : E \to \C$, and let $'\lambda = {}^\iota\lambda$, to which using the statement over $\C$ one gets $'\kappa$, which defines a unique $\kappa$ 
via ${}^\iota\kappa =  {}'\kappa.$ It is clear that $\lambda = {\rm BC}_{F/F_1}(\kappa)$ because this is so after applying $\iota$. 
To prove the statement over $\C$, take $\lambda \in X^+_{00}(\Res_{F/\Q}(T_0) \times \C)$, and suppose $\lambda = (\lambda^\eta)_{\eta : F \to \C}$ with 
$\lambda^\eta = (b^\eta_1 \geq b^\eta_2 \geq \cdots \geq b^\eta_n).$ Strong-purity gives 
$$
b^{\gamma \circ \eta}_{n-j+1} + b^{\gamma \circ \bar{\eta}}_j = \w, \quad \forall \gamma \in \Gal(\bar\Q/\Q), \ \forall \eta \in \Sigma_F, \ 1 \leq j \leq n.
$$
Also, one has: 
$$
b^{\gamma \circ \eta}_{n-j+1} + b^{\overline{\gamma \circ \eta}}_j = \w, \quad \forall \gamma \in \Gal(\bar\Q/\Q), \ \forall \eta \in \Sigma_F, \ 1 \leq j \leq n.
$$
Hence, we get 
$b^{\gamma \circ \bar{\eta}}_j \ = \ b^{\overline{\gamma \circ \eta}}_j.$ This may be written as $b^{\gamma \circ \c \circ \eta}_j \ = \ b^{\c \circ \gamma \circ \eta}_j,$ where 
$\c$ is complex conjugation. Hence, as explicated in the proof of Prop.\,26 in \cite{raghuram-hecke}, one gets $b^{\gamma \circ \eta}_j = b^{\eta}_j$ for all 
$\gamma$ in the normal subgroup of $\Gal(\bar\Q/\Q)$ generated by the commutators $\{g \c g^{-1} \c : g \in \Gal(\bar\Q/\Q)\}$, and all $\eta : F \to \C.$ 
This means that $b^{\eta}_j$ depends 
only on $\eta|_{F_1}.$ 
\end{proof}

\medskip
\subsection{\bf The nonvanishing problem of cuspidal cohomology for $\GL(n)$}
\label{sec:main-prob}

The basic conjecture that any strongly-pure weight for $\GL_n$ over a number field $F$ supports cuspidal cohomology is formalized as: 

\begin{con}
\label{con:main}
Suppose $\lambda \in X^+_{00}(\Res_{F/\Q}(T_0) \times E)$ is a strongly-pure weight, then for some deep enough open-compact subgroup $K_f$ of $G(\A_f)$, 
$$
H_{\rm cusp}^\bullet(\SGK, \tM_{{}^\iota\lambda, \C}) \neq 0,
$$
for every embedding $\iota : E \to \C$.  
\end{con}

The main theorem of this article settles the conjecture when $F$ is Galois over its maximal totally real subfield $F_0$; this includes the case when $F/\Q$ is a Galois extension.

\bigskip

\section{Cuspidal cohomology of $\SL(n)$ and $\GL(n)$}
\label{sec:gln-sln}
As a preliminary reduction step, we show using 
the main result of Labesse--Schwermer \cite{labesse-schwermer-JNT}, 
that the nonvanishing of cuspidal cohomology for $\GL(n)$ with coefficients in $\M_{{}^\iota\lambda, \C}$ 
follows from the nonvanishing of cuspidal cohomology for $\SL(n)$ with coefficients in the restriction of $\M_{{}^\iota\lambda, \C}$ to $\SL(n)$. 

\medskip
Towards this reduction step, denote 
$G_0^1= \SL_n/F$, $B_0^1 = B_0 \cap G_0^1,$ and $T_0^1 = T_0 \cap G_0^1$. Let $G^1,$ $B^1,$ and $T^1$ be the Weil restriction of scalars from $F$ to $\mathbb{Q}$ of 
$G_0^1,$ $B_0^1,$ and $T_0^1,$ respectively. Then, $G^1(\R) \simeq \prod_{v \in S_r} \SL_n(\R) \times \prod_{v \in S_c} \SL_n(\C)$; its Lie algebra is denoted $\g_\infty^1,$
and its maximal compact subgroup is 
$C_\infty^1 \simeq \prod_{v \in S_r} \SO(n) \times \prod_{v \in S_c} \SU(n).$ 
Suppose $\lambda \in X^*(\Res_{F/\Q}(T_0) \times E)$ and $\iota : E \to \C$ is an embedding. Let ${}^\iota\lambda^1$ denote
the restriction of ${}^\iota\lambda$ to a character of $T^1$. Since the $\iota$-conjugate of the restriction of $\lambda$ to $T^1$ is the same as restriction to $T^1$ of the $\iota$-conjugate of 
$\lambda$, the notation ${}^\iota\lambda^1$ will not cause any confusion. 
For an open-compact subgroup $K_f^1$ of $G^1(\A_f)$, the definitions for the locally symmetric space $\cS^{G^1}_{K_f^1} = G^1(\Q)\backslash G^1(\A)/C_\infty^1K_f^1$, 
the sheaf $\tM_{{}^\iota\lambda^1,\C}$ of $\C$-vector spaces on $S^{G^1}_{K_f^1}$, its cohomology 
groups $H^\bullet(S^{G^1}_{K_f^1}, \tM_{{}^\iota\lambda^1, \C})$, and cuspidal cohomology $H_{\rm cusp}^\bullet(S^{G^1}_{K_f^1}, \tM_{{}^\iota\lambda^1, \C})$ works just the same
for $\SL(n)$ as for $\GL(n)$ in Sect.\,\ref{sec:loc-symm-space} through Sect.\,\ref{sec:cuspidal-coh}.  
The compact-mod-center groups that we divide by in the definition of the locally symmetric spaces $\SGK$ and $\cS^{G^1}_{K_f^1}$ are related as: 
$$
K_\infty^\circ \ = \ C_\infty^\circ \cdot S(\R)^\circ \ = \  C_\infty^1 \cdot (S(\R)^\circ \cdot \prod_{v \in S_c} Z_{\U(n)}),
$$
where $Z_{\U(n)}$ is the center of $\U(n)$ the maximal compact subgroup of $\GL_n(F_v)$ for $v \in S_c$. 
The following proposition is a mild generalization of one part of the main theorem of Labesse--Schwermer \cite[Thm.\,1.1.1]{labesse-schwermer-JNT} applied to 
$\SL(n)$ and $\GL(n)$.

\medskip

\begin{prop}
\label{prop:gln-sln-variation}
Suppose $\sigma$ is a unitary cuspidal automorphic representation of $G^1(\A) = \SL_n(\A_F)$. Then there is 
a unitary cuspidal automorphic representation $\pi$ of $G(\A) = \GL_n(\A_F)$ such that 
\begin{enumerate}
\item[(i)] $\sigma$ appears in the restriction of $\pi$, and 
\item[(ii)] the central character $\omega$ of $\pi$ is trivial on $S(\R)^\circ \prod_{v \in S_c} Z_{\U(n)}$. 
\end{enumerate}
\end{prop}

\begin{proof}
That there exists a unitary cuspidal $\pi$ whose restriction to $\SL_n(\A_F)$ contains $\sigma$ is contained in \cite[Thm.\,1.1.1]{labesse-schwermer-JNT}. The only additional detail is that we may 
arrange for the condition on the central character $\pi$; this follows from the proof of their theorem--see Sect.\,5.1 and 5.2 of {\it loc.\,cit.} We adumbrate the details below. 
Consider 
$$
L^2_{\rm cusp}(\SL_n(F)\backslash \SL_n(\A_F), \omega_0),
$$
where $\omega_0 = \omega_\sigma$ is the central character of $\sigma$; it is a character on the center of $\SL_n(\A_F)$--that this center is disconnected will no affect the discussion. Let 
$W_{\sigma}$ be the isotypic component of $\sigma$ in the above space of cuspforms. Let 
$$
Z_1 \ := \ (\GL_n(F) Z_0(\A_F)) \cap \SL_n(\A_F),
$$
where, recall that $Z_0$ is the center of $\GL_n/F$. Decompose $W_\sigma$ according to the action of the abelian group $\SL_n(F)\backslash Z_1$; let $\omega_1$ be a character of $\SL_n(F)\backslash Z_1$ that occurs and consider the space $W_\sigma(\omega_1)$ cut out by this character. Observe that
\begin{equation}
\label{eqn:inclusion_centers}
\SL_n(F)\backslash Z_1 \ \subset \ \GL_n(F)\backslash (\GL_n(F) Z_0(\A_F)) \ \simeq \ Z_0(F)\backslash Z_0(\A_F).
\end{equation}
Let $\omega$ be a character of $Z_0(F)\backslash Z_0(\A_F)$ that extends $\omega_1$ on $\SL_n(F)\backslash Z_1.$ Then Labesse and Schwermer prove that 
$$
L^2_{\rm cusp}(\SL_n(F)\backslash \SL_n(\A_F), \omega_1) \ \simeq \ 
L^2_{\rm cusp}(\GL_n(F)\backslash H^+, \omega),
$$
where $H^+ = \GL_n(F) Z_0(\A_F)\SL_n(\A_F)$, and that 
$$
L^2_{\rm cusp}(\GL_n(F)\backslash \GL_n(\A_F), \omega) \ \simeq \
\Ind_{H^+}^{ \GL_n(\A_F)} \left(L^2_{\rm cusp}(\GL_n(F)\backslash H^+, \omega)\right). 
$$
A Mackey type argument permits them to find $\pi$ in $L^2_{\rm cusp}(\GL_n(F)\backslash \GL_n(\A_F), \omega)$ extending the original $\sigma$, and of course the central 
character of $\pi$ is $\omega$. The proposition is an exercise in choosing an $\omega$ using \eqref{eqn:inclusion_centers}. Since
$$
Z_1 \cap \left(S(\R)^\circ \prod_{v \in S_c} Z_{\U(n)}\right) \ = \
(\GL_n(F) Z_0(\A_F)) \cap \SL_n(\A_F) \cap \left(S(\R)^\circ \prod_{v \in S_c} Z_{\U(n)}\right) \ = \{1\},
$$
we can first extend $\omega_1$ to a character $\omega_2$ on $Z_2 := Z_1 \cdot (S(\R)^\circ \prod_{v \in S_c} Z_{\U(n)})$ by making it trivial on $S(\R)^\circ (\prod_{v \in S_c} Z_{\U(n)});$ 
next, via classical Mackey theory, extend $\omega_2$ to a character $\omega$ on  $Z_0(F)\backslash Z_0(\A_F),$ and then run through the rest of the argument of Labesse and Schwermer. 
\end{proof}

For an automorphic representation $\pi$ of $\GL_n(\A_F)$, and any real number 
$t$, let $\pi(t)$ stand for $\pi \otimes |\!| \ |\!|_F^t$, where $|\!| \ |\!|_F$ is the adelic norm on $\A_F^\times$ which is a product of local normalized valuations; a similar notation is also adopted for local representations. The main result of this section is the following result:

\begin{prop}
\label{prop:gln-sln}
Suppose $\lambda \in X^+_{00}(\Res_{F/\Q}(T_0) \times E)$ with purity weight $\w$, 
and $\iota : E \to \C$ is an embedding. Suppose 
$H_{\rm cusp}^\bullet(S^{G^1}_{K_f^1}, \tM_{{}^\iota\lambda^1, \C}) \neq 0,$ 
for some open-compact subgroup $K_f^1$ of $G^1(\A_f)$ and suppose 
$\sigma$ is a unitary cuspidal automorphic representation of $G^1(\A) = \SL_n(\A_F)$ which contributes to $H_{\rm cusp}^\bullet(S^{G^1}_{K_f^1}, \tM_{{}^\iota\lambda^1, \C}).$ 
Suppose $\pi$ is a unitary cuspidal automorphic representation of $G(\A) = \GL_n(\A_F)$ such that $\sigma$ appears in the restriction of $\pi$ (\cite[Thm.\,1.1.1]{labesse-schwermer-JNT})
and satisfying the central character condition as in Prop.\,\ref{prop:gln-sln-variation}.  

Then, $\pi(-\w/2)$ contributes nontrivially to $H_{\rm cusp}^\bullet(\SGK, \tM_{{}^\iota\lambda, \C}) \neq 0,$ 
for some open-compact subgroup $K_f$ of $G(\A_f)$. (Here, $K_f^1$ need not be equal to $K_f \cap G^1(\A_f)$.)
\end{prop}

\begin{proof}
The archimedean component $\sigma_\infty$ of the cuspidal $\sigma$ of $G^1(\A) = \SL_n(\A_F)$ satisfies: 
\begin{equation}
\label{eq:sln-coh}
 H^\bullet(\fg_\infty^1, C_\infty^1; \sigma_\infty \otimes \M_{{}^\iota\lambda^1, \C}) \neq 0.
\end{equation}
Given a unitary cuspidal automorphic representation $\pi$ of $\GL_n(\A_F)$ whose restriction contains $\sigma$ as in as in Prop.\,\ref{prop:gln-sln-variation}, 
 it suffices to show that \eqref{eq:sln-coh} implies the following nonvanishing: 
\begin{equation}
\label{eq:gln-coh}
 H^\bullet(\fg_\infty, K_\infty^0; \pi_\infty(-\w/2) \otimes \M_{{}^\iota\lambda, \C}) \neq 0.
\end{equation}
The embedding $\iota$ is fixed for the entire proof, and for brevity it will be dropped; in effect ${}^\iota\lambda$ is replaced simply by $\lambda$; this abuse of notation will cause 
no confusion. 

\medskip

From \eqref{eq:sln-coh}, applying the K\"unneth theorem for relative Lie algebra cohomology, for each $v \in S_\infty$, one has: 
\begin{equation}
\label{eq:sln-coh-local}
 H^\bullet(\fg_v^1, C_v^1; \sigma_v \otimes \M_{\lambda^1_v, \C}) \neq 0.
\end{equation}
For $v \in S_r$, $\M_{\lambda^1_v, \C}$ is the restriction to $\SL_n(\C)$ of the irreducible representation $\M_{\lambda^{\eta_v}, \C}$ of $\GL_n(\C)$ of highest weight $\lambda^{\eta_v}$; 
however, for $v \in S_c$, $\M_{\lambda^1_v, \C}$ is the restriction to $\SL_n(\C)$ of the representation $\M_{\lambda^{\eta_v}, \C} \otimes \M_{\lambda^{\overline{\eta_v}}, \C}$ 
of $\GL_n(\C)$, where we remind ourselves that an element of $\GL_n(\C)$ acts on the second factor after taking its complex conjugate. 

\medskip

Let $\pi_v^1$ denote the restriction of $\pi_v$ to $\SL_n(F_v)$. 
If $v \in S_c$ or if $n$ is odd then $\pi_v^1$ is irreducible and hence $\pi_v^1 =\sigma_v$. 
If $v \in S_r$ and $n$ is even then $\pi_v^1$ is either irreducible or a direct sum $\pi_v^+ \oplus \pi_v^-$ of two inequivalent irreducible 
representations; the non-trivial element of $\rO(n)/\SO(n)$ conjugates $\pi_v^+$ to $\pi_v^-,$ and 
induces an isomorphism between $H^\bullet(\fg^1_v, C_v^1; \pi_v^+ \otimes \M_{{}^\iota\lambda^1_v, \C})$ and 
$H^\bullet(\fg^1_v, C_v^1; \pi_v^- \otimes \M_{{}^\iota\lambda^1_v, \C});$ hence, the dimension of 
$H^\bullet(\fg^1_v, C_v^1; \pi^1_v \otimes \M_{{}^\iota\lambda^1_v, \C})$ is twice that of $H^\bullet(\fg^1_v, C_v^1; \sigma_v \otimes \M_{{}^\iota\lambda^1_v, \C}).$ 
 From the K\"unneth theorem, we get 
\begin{equation}
\label{eq:gln_coh-1}
H^\bullet(\fg_\infty^1, C_\infty^1; \pi_\infty^1 \otimes \M_{{}^\iota\lambda^1, \C}) \neq 0.
\end{equation}
This cohomology, by definition, is the cohomology of the complex
$$
\Hom_{C_\infty^1} \left(\wedge^\bullet(\fg_\infty^1/\c_\infty^1), \pi_\infty^1 \otimes \M_{{}^\iota\lambda^1, \C}\right). 
$$
Now as a $K_\infty^\circ$-module, via the Adjoint-action, we have
$$
\fg_\infty/\k_\infty \ = \ (\fg_\infty^1/\c_\infty^1) \oplus (\z_\infty/\s_\infty),
$$
where its action on $\z_\infty/\s_\infty$ is trivial. Taking the exterior algebra on both sides one has:
$$
\wedge^\bullet(\fg_\infty/\k_\infty) \ = \  \wedge^\bullet(\fg_\infty^1/\c_\infty^1) \otimes \wedge^\bullet(\z_\infty/\s_\infty). 
$$
We claim that this induces an isomorphism of the complexes:
\begin{multline}
\label{eq:iso-complex}
\Hom_{K_\infty^\circ} \left(\wedge^\bullet(\fg_\infty/\k_\infty), \pi_\infty(-\w/2) \otimes \M_{{}^\iota\lambda, \C}\right) \ \simeq \\ 
\Hom_{C_\infty^1} \left(\wedge^\bullet(\fg_\infty^1/\c_\infty^1), \pi_\infty^1 \otimes \M_{{}^\iota\lambda^1, \C}\right) \otimes 
\wedge^\bullet(\z_\infty/\s_\infty). 
\end{multline}
Taking cohomology on both sides of \eqref{eq:iso-complex}, using \eqref{eq:gln_coh-1}, we deduce \eqref{eq:gln-coh}.

\medskip

It suffices to justify the claim in \eqref{eq:iso-complex}. Since 
$K_\infty^\circ = C_\infty^1 \cdot (S(\R)^\circ \cdot \prod_{v \in S_c} Z_{\U(n)}),$ an operator in the left hand side of 
\eqref{eq:iso-complex}, which, by definition is $K_\infty^\circ$-equivariant, is also $C_\infty^1$-equivariant. As a $C_\infty^1$-module,  
$\pi_\infty(-\w/2) \otimes \M_{{}^\iota\lambda, \C}$ is the same as $\pi_\infty^1 \otimes \M_{{}^\iota\lambda^1, \C}$. Whence, there is a 
natural inclusion of the left hand side of \eqref{eq:iso-complex} into the right hand side. For the converse, suppose 
$T \in \Hom_{C_\infty^1} \left(\wedge^\bullet(\fg_\infty/\k_\infty), \pi_\infty^1 \otimes \M_{{}^\iota\lambda^1, \C}\right)$, then  
$K_\infty^\circ$-equivariance of $T$ follows if we show that, furthermore, $T$ is equivariant for $S(\R)^\circ \cdot \prod_{v \in S_c} Z_{\U(n)}$. 
The adjoint-action of 
$S(\R)^\circ \cdot \prod_{v \in S_c} Z_{\U(n)}$ on $(\fg_\infty/\k_\infty)$ is trivial, and hence also trivial on $\wedge^\bullet(\fg_\infty/\k_\infty)$. 
It suffices then to show that $S(\R)^\circ \cdot \prod_{v \in S_c} Z_{\U(n)}$ acts trivially on $\pi_\infty(-\w/2) \otimes \M_{{}^\iota\lambda, \C}.$
Recall from Sect.\,\ref{sec:strong-pure-C} the purity of $\lambda$; if $\lambda = (\lambda^{\eta})_{\eta : F \to \C}$ and 
$\lambda^{\eta} = \sum_{i = 1}^{n-1} (a^\eta_i - 1) \bfgreek{gamma}_i + d^\eta \cdot \bfgreek{delta}$, then purity implies $d^\eta = \w/2$ for all $\eta$; permitting us to write 
$\lambda^\eta = \lambda^{\eta, 1} + \w/2 \cdot \det.$ For $v \in S_c$, let ${\sf z}_v = {\rm diag}(z_v, \dots, z_v) \in Z_{\U(n)}$ for some 
$z_v \in \C^\times$ with $|z_v|_v = z_v \bar{z}_v = 1;$ let ${\sf t} = {\rm diag}(t, \dots, t)  \in S(\R)^\circ$; 
the action of element 
${\sf t} \cdot ({\sf z}_v)_{v \in S_c}  \ \in \ S(\R)^\circ \cdot \prod_{v \in S_c} Z_{\U(n)}$
on $\pi_\infty(-\w/2) \otimes \M_{{}^\iota\lambda, \C}$ is via the scalar 
\begin{multline*}
 \left(\omega_{\pi}({\sf t}) \cdot |\!| t |\!|_\infty^{-n\w/2 + \sum_{\eta: F \to \C} nd^{\eta}} \right) \cdot
\prod_{v \in S_c} \left(\omega_{\pi_v}(z_v) \cdot (z_v \bar{z}_v)^{-n\w/2} \cdot (z_v^{n d^\eta} \bar{z}_v^{n d^{\bar\eta}})\right) \\ 
\ = \ \omega_{\pi}({\sf t} \cdot ({\sf z}_v)_{v \in S_c} ) \ = \ 1, 
\end{multline*}
because $d^\eta = \w/2$ for all $\eta$ giving the first equality, and the second equality follows from the condition imposed on the central character as in Prop.\,\ref{prop:gln-sln-variation}. This gives
the reverse inclusion \eqref{eq:iso-complex}, completing the proof.
\end{proof}

\medskip

\section{A result of Borel, Labesse, and Schwermer}
\label{sec:BLS}

In this section we review the main theorem of Borel--Labesse--Schwermer \cite[Thm.\,10.4]{borel-labesse-schwermer} which will be used in the proof of Thm.\,\ref{thm:main-sln} in the 
following section.

\subsection{Review of cuspidal cohomology as in \cite{borel-labesse-schwermer}}
Just for this section, following the context as set up in \cite{borel-labesse-schwermer}, suppose $G$ is an almost absolutely-simple connected algebraic group over a number field $F$ of strictly positive $F$-rank. Let $(\phi, \M)$ be a finite-dimensional irreducible complex representation of $G_\infty$. For a finite set of places  $S$ of $F$, let $G_S = \prod_{v \in S} G(F_v)$. 
Suppose now that $S$ contains all the archimedean places; write $S = S_\infty \cup S_f$; extend $\M$ to a representation of $G_S$ by making it trivial on $G_{S_f}.$ 
Then $\M = \otimes_{v \in S} \M_v$, with $M_v$ an irreducible representation of $G(F_v)$; it is the trivial representation for $v \in S_f$. 
Let $\Gamma$ be an $S$-arithmetic subgroup of $G(F)$. The cuspidal cohomology of $\Gamma$ with coefficients in $\M$ is defined as
$$
H^\bullet_{\rm cusp}(\Gamma, \M) \ := \ H_d^\bullet(G_S, L^2_{\rm cusp}(\Gamma \backslash G_S)^\infty \otimes \M),
$$
where $H_d^\bullet(G_S,-)$ is differentiable cohomology; see Borel--Wallach \cite[Chap.\,IX]{borel-wallach}. To get information about nonvanishing of cuspidal cohomology 
for deep enough $S$-arithmetic subgroups, one may pass to the limit to define:
$$
H^\bullet_{\rm cusp}(G, S; \M)
 \ := \ 
\lim_{\longrightarrow_\Gamma} H^\bullet_{\rm cusp}(\Gamma, \M),
$$
which is isomorphic to the differential cohomology of adelic cuspidal spectrum: 
$$
H^\bullet_{\rm cusp}(G, S; \M) \ = \  
H_d^\bullet(G_S, L^2_{\rm cusp}(G(F) \backslash G(\A_F))^\infty \otimes \M).
$$

\medskip

The van Est theorem (see, for example, \cite[Cor.\,IX.5.6, (ii)]{borel-wallach}) says that for any differentiable $G_\infty$-module $V$, if $K_\infty$ is a maximal compact subgroup of $G_\infty,$
then the differentiable cohomology of $V$ is the same as its relative Lie algebra cohomology, i.e., $H^\bullet_d(G_\infty, V) = H^\bullet(\g_\infty, K_\infty, V).$ In particular, taking 
$S = S_\infty$, we see that 
$$
H^\bullet_{\rm cusp}(G, S_\infty; \M) \ = \  
H^\bullet(\g_\infty, K_\infty, L^2_{\rm cusp}(G(F) \backslash G(\A_F))^\infty \otimes \M),
$$
which ties up with our previous definition of cuspidal cohomology $H_{\rm cusp}^\bullet(G, \M)$ as in \eqref{eqn:def-cusp-coh-no-Kf} after adapting the context therein which was $\GL_n/F$ to the $G/F$ in this section.

\medskip
\subsection{Lefschetz number at infinity of a rational automorphism}
This subsection is adapted from \cite[Sect.\,8.1]{borel-labesse-schwermer}. 
Let $\alpha$ be an automorphism of the algebraic group $G/F$ of finite order. Define $\tilde{L} = G \rtimes \langle\alpha\rangle$, where $\langle\alpha\rangle$ is a finite cyclic group considered 
as an $F$-group with every element being $F$-rational. Let $L$ denote the coset defined by $\alpha.$ For any $F$-algebra $A$, 
denote by $L(A)^+$ the group generated by $L(A).$ 

\medskip

At any place $v$ of $F$, suppose $\M$ is a finite-dimensional representation of $L(F_v)^+$ which is taken to be trivial if $v$ is a finite place. 
Let $(\pi,H_{\pi})$ be an irreducible admissible representation of $L(F_v)^+$. It is helpful to keep in mind that an irreducible representation $\pi$ of $L(F_v)^+$ is the same as an irreducible representation $\pi_0$ of $G(F_v)$ that is $\alpha$-invariant: ${}^\alpha\pi_0 \simeq \pi_0$, a choice of an equivariant map determines an extension of $\pi_0$ on $G(F_v)$ to $\pi$ on $L(F_v)^+$. 
Let $\beta\in L(F_v).$ By an abuse of notation, let $\beta$ also denote the 
automorphism it induces on $H^*_d(G(F_v), H_{\pi}\otimes \mathcal{M})$. The Lefschetz number of $\beta$ with respect to $H_{\pi}\otimes \mathcal{M}$ is defined as 
$$
\Lef(\beta,G(F_v),H_{\pi}\otimes \mathcal{M}) \ := \ \sum_m (-1)^m \, {\rm Trace}\left(\beta: H^m_d(G(F_v), H_{\pi}\otimes \mathcal{M})\right).
$$ 
Since $G(F_v)$ acts trivially on the differentiable cohomology groups, the Lefschetz number $\Lef(\beta,G(F_v),H_{\pi}\otimes \mathcal{M})$ is independent of the choice of 
element representing $\beta \in L(F_v)$, and we may and shall denote it by $\Lef(\alpha,G(F_v),H_{\pi}\otimes \mathcal{M})$. 

\medskip

Let $S$ be a finite set of places containing all the archimedean places. Let $\mathcal{M}_{v}$ and $H_{\pi, v}$ be representations of $G(F_v)$ as above. Then 
$H_{\pi}=\otimes_{v \in S} H_{\pi,v}$ and $\mathcal{M} =\otimes_{v \in S} \mathcal{M}_{v}$ are representations of $G_S$. Define the $S$-local Lefschetz number of $\alpha$ with respect to $H_{\pi}\otimes \mathcal{M}$ as: 
$$
\Lef(\alpha,G_S,H_{\pi}\otimes \mathcal{M}) \ := \ \prod_{v \in S} \Lef(\alpha, G(F_v),H_{\pi,v}\otimes \mathcal{M}_v).
$$
The $S_\infty$-local Lefschetz number of $\alpha$ is called the Lefschetz number of $\alpha$ at infinity.

\medskip
\subsection{A criterion for nonvanishing of cuspidal cohomology}

\begin{thm}[Thm.\,10.4 of Borel--Labesse--Schwermer \cite{borel-labesse-schwermer}]
\label{thm:bls}
Let $G$ be an almost absolutely-simple connected algebraic group over a number field $F$ of strictly positive $F$-rank. 
Assume that the Lefschetz number of $\alpha$ at infinity as a function
\[
\pi_\infty \ \mapsto \ \Lef(\alpha, G_\infty; \, H_{\pi_\infty} \otimes \M), 
\]
with $\pi_\infty$ varying through the set of equivalence classes of irreducible unitary representations of the Lie group $G_\infty,$ is not identically zero. Then for a finite set 
$S$ of places containing all archimedean places and at least one finite place, the 
cuspidal cohomology $H^\bullet_{\rm cusp}(G, S; \M)$ does not vanish. 
\end{thm}

Suppose $S = S_\infty \cup S_f$ as before, and $r_f = \sum_{v \in S_f} r_v$, the sum being over $v \in S_f$ of the $F_v$-rank $r_v$ of $G/F_v$, then by Thm.\,6.5 of \cite{borel-labesse-schwermer} one has
an injection:
\begin{equation}
\label{eqn:S-cuspidal-implies-cuspidal}
H^\bullet_d(G_S, L^2_{\rm cusp}(G(F) \backslash G(\A_F))^\infty \otimes \M)[-r_f] \ \subset \ 
H^\bullet_d(G_\infty, L^2_{\rm cusp}(G(F) \backslash G(\A_F))^\infty \otimes \M).
\end{equation}
This is due to the fact that for finite places $v \in S_f$, $\M_v$ is the trivial representation, and the only representation of $G(F_v)$ that has nontrivial cohomology (with respect to 
the trivial representation) is the Steinberg representation of $G(F_v)$ which has nontrivial cohomology only in degree $r_v.$ The theorem then implies that 
$H^\bullet_{\rm cusp}(G, S_\infty; \M) \neq 0.$ 



\medskip
\section{The main theorem and its proof}
\label{sec:main_theorem}

\medskip
\subsection{Statement of the main theorem on nonvanishing of cuspidal cohomology}
\label{sec:main_theorem_statement}

\begin{thm}
\label{thm:main-gln}
Let $F$ be a number field. Let $F_0$ be the maximal totally real subfield of $F$.  
Assume that $F/F_0$ is a Galois extension. (This is automatically satisfied if we assume that $F/\Q$ is a Galois extension.) 
Let $G_0 = \GL_n/F$, $T_0$ a maximal torus in $G_0$, and all other notations as in Sect.\,\ref{sec:prelims-gln} and Sect.\,\ref{sec:pure}. 
Suppose $\lambda \in X^+_{00}(\Res_{F/\Q}(T_0) \times E)$ is a strongly-pure weight, then we have nonvanishing of cuspidal cohomology for some level structure, i.e., 
for some deep enough open-compact subgroup $K_f$ of $G(\A_f)$ we have 
$$H_{\rm cusp}^\bullet(\SGK, \tM_{{}^\iota\lambda, \C}) \neq 0$$ 
for every embedding $\iota : E \to \C$.  
\end{thm}

\medskip

If $F$ is totally real, i.e., $F = F_0$, then this nonvanishing of cuspidal cohomology was proved in Bhagwat--Raghuram \cite{bhagwat-raghuram-BIMS} by using Arthur's classification and transferring from a suitable cuspidal representation of an appropriate classical group to $G$ whose choice depends on the parity of $n$ and the parity of the purity weight $\w$ of the highest weight $\lambda.$ If $F$ itself is a CM-field, then as in \cite{bhagwat-raghuram-BIMS}, in many cases, one has nonvanishing of cuspidal cohomology by transferring from a suitable unitary group. However, the proof of Thm.\,\ref{thm:main-gln} below does not appeal to \cite{bhagwat-raghuram-BIMS} and gives a new proof of nonvanishing of cuspidal cohomology even when $F$ is a totally real field or a CM field. 
If $F$ is not totally real, the hypothesis $F/F_0$ is Galois does not imply that $F$ is totally imaginary; for example, $F_0$ a real quadratic extension and take an $a \in F_0$ which 
is neither totally positive nor totally negative and put $F = F_0(\sqrt{a}).$ Also, if $F$ is totally imaginary, it is worth emphasizing that $F$ need not be a CM field; in fact, $F$ need not even contain a CM-subfield. 

\medskip

To prove the theorem, after appealing to the reduction step in Prop.\,\ref{prop:gln-sln}, it suffices to prove the following theorem on nonvanishing of cuspidal cohomology for 
$\SL_n/F$: 

\medskip

\begin{thm}
\label{thm:main-sln}
Let $F$ be a number field which is Galois over its maximal totally real subfield $F_0$. 
Let $G_0^1 = \SL_n/F$ and $G^1 = \Res_{F/\Q}(G_0^1)$; the rest of the notations are as in Sect.\,\ref{sec:gln-sln}. 
Suppose $\lambda \in X^+_{00}(\Res_{F/\Q}(T_0) \times E)$ is a strongly-pure weight, then 
for some open-compact subgroup $K_f^1$ of $G^1(\A_f)$ we have $$H_{\rm cusp}^\bullet(S^{G^1}_{K_f^1}, \tM_{{}^\iota\lambda^1, \C}) \neq 0$$ 
for every embedding $\iota : E \to \C$.  
\end{thm}

\medskip
\subsection{Proof of Thm.\,\ref{thm:main-sln}}
To begin, the local computations are put together in \ref{sec:slnc} for $\SL_n(\C)$  and in \ref{sec:slnr} for $\SL_n(\R)$; 
the key technical ingredients of the nonvanishing of archimedean Lefschetz numbers are 
contained in Prop.\,\ref{prop:slnc-lefschetz} and Prop.\,\ref{prop:slnr-lefschetz}. 
Then the global proof in \ref{sec:slnF} is an application of Thm.\,\ref{thm:bls}.

\medskip
\subsubsection{\bf Lefschetz number for $\SL_n(\C)$}
\label{sec:slnc}
The main result of this paragraph is contained in Prop.\,\ref{prop:slnc-lefschetz}. Its proof uses Vogan--Zuckerman \cite{vogan-zuckerman} on the inner structure of cohomology. 
The entire discussion being local, just for this paragraph, the notations are 
adapted to $\SL_n(\C)$.

\medskip
\paragraph{\it Notations}
\label{par:notations}
Let $G = \GL_n(\C)$, and $G^1 = \SL_n(\C)$. Let $K^1 = \SU(n)$ be the standard maximal compact subgroup of $G^1$. Let $\theta(g)= {}^t \bar{g}^{-1}$ for $g \in G$; it is an automorphism of order $2$, whose fixed point set within $G^1$ is $K^1$. Let $\g = \gl_n(\C)$ (resp., $\g^1 = \sl_n(\C)$) be the Lie algebra of the Lie group $G$ (resp., $G^1$) on which the action of $\theta$, which we continue to denote $\theta$, is via $\theta(X) = - {}^t \!\bar{X}$ for $X \in \g.$ On $\g^1$, the $+1$ eigenspace of $\theta$ is $\k^1$, the Lie algebra of $K^1,$ 
consisting of all skew-hermitian matrices in $\sl_n(\C)$, and the $-1$ eigenspace of $\theta$, denoted $\p^1$, consists of all hermitian matrices; one has $\g^1 = \k^1 \oplus \p^1,$ a decomposition 
of real Lie algebras.  
Denote by $\g^1_c$, $\k^1_c$, and $\p^1_c$ the complexifications of $\g^1$, $\k^1,$ and $\p^1,$ respectively. Identify $\g^1_c/\k^1_c$ with $\p^1_c$. 
Let $B$ (resp., $B^1$) denote the Borel subgroup of upper-triangular
matrices in $G$ (resp., $G^1$), with Levi decomposition $B = T \cdot N$ (resp., $B^1 = T^1 \cdot N$) where $T$ (resp., $T^1$) is the subgroup of diagonal matrices in $G$ (resp., $G^1$), 
and $N$ the subgroup of upper-triangular unipotent matrices in $G$.
Then ${}^\circ T^1 := K^1 \cap T^1$ is the maximal torus of the compact Lie group $K^1$ consisting of all diagonal matrices 
${\rm diag}(z_1, z_2, \dots, z_n),$ with  $z_1,\dots, z_n \in S^1,$ the circle group in $\C^*$, and
$\prod_{j=1}^n z_j = 1.$ We have the Langlands decomposition $B^1 = {}^\circ T^1 \cdot A^1 \cdot N,$ where $A^1$ consists of all diagonal matrices ${\rm diag}(a_1, a_2,\dots, a_n),$ with $a_1,\dots, a_n \in \R_{> 0},$ and 
$\prod_{j=1}^n a_j = 1.$ The Lie algebras over $\R$ of the real Lie groups 
$B^1,$ ${}^\circ T^1,$ $A^1,$ and $N$ are denoted $\b^1$, ${}^\circ\t^1$, $\a^1,$ and $\n$, respectively, and for their complexification we add the subscript $c$. Over $\R$ one has 
$$
\k^1 \ = \ {}^\circ\t^1  \oplus 
\left( \bigoplus_{1 \leq i < j \leq n} \R \cdot (E_{ij} - E_{ji}) \right) \oplus 
 \left( \bigoplus_{1 \leq i < j \leq n} \R \cdot \sqrt{-1}(E_{ij} + E_{ji}) \right),
$$
where $E_{ij}$ is the matrix with $1$ in the $(i,j)$-th entry and $0$ elsewhere; similarly, 
$$
\p^1 \ = \ \a^1  \oplus 
\left( \bigoplus_{1 \leq i < j \leq n} \R \cdot (E_{ij} + E_{ji}) \right) \oplus 
 \left( \bigoplus_{1 \leq i < j \leq n} \R \cdot \sqrt{-1}(E_{ij} - E_{ji}) \right). 
$$
After complexifying, and diagonalizing the adjoint action of ${}^\circ\t_c$, one has the decomposition:
$$
\k^1_c \ = \ {}^\circ\t^1_c \oplus \left( \bigoplus_{\alpha \in \Phi^+} \C \cdot X_\alpha^k \right) \oplus  \left( \bigoplus_{\beta \in \Phi^-} \C \cdot Y_\beta^k \right),
$$
where, $\Phi^+$ is the set of positive roots identified with pairs of indices $1 \leq i < j \leq n$ in the usual way, $\Phi^- = -\Phi^+$ is the set of negative roots, 
and as elements of $\k^1_c = \k^1 \otimes \C$, one has: 
$$
X_{ij}^k \ = \ (E_{ij} - E_{ji}) \otimes \sqrt{-1}  \ +  \ \sqrt{-1}(E_{ij} + E_{ji}) \otimes 1, 
$$
$$
Y_{ij}^k \ = \ (E_{ij} - E_{ji}) \otimes \sqrt{-1}  \ -  \ \sqrt{-1}(E_{ij} + E_{ji}) \otimes 1.  
$$
Similarly, 
$$
\p^1_c \ = \ \a^1_c \oplus \left( \bigoplus_{\alpha \in \Phi^+} \C \cdot X_\alpha^p \right) \oplus  \left( \bigoplus_{\beta \in \Phi^-} \C \cdot Y_\beta^p \right),
$$
where, as elements of $\p^1_c = \p^1 \otimes \C$, one has: 
$$
X_{ij}^p \ = \ (E_{ij} + E_{ji}) \otimes \sqrt{-1}  \ +  \ \sqrt{-1}(E_{ij} - E_{ji}) \otimes 1, 
$$
$$
Y_{ij}^p \ = \ (E_{ij} + E_{ji}) \otimes \sqrt{-1}  \ -  \ \sqrt{-1}(E_{ij} - E_{ji}) \otimes 1.  
$$
The above decomposition will also be written as 
$$
\p^1_c = \a^1_c \oplus \u_c, 
$$ 
with $\u_c = \left(\oplus_{\alpha \in \Phi^+} \C \cdot X_\alpha^p \right) \oplus  \left( \oplus_{\beta \in \Phi^-} \C \cdot Y_\beta^p \right).$
Let $\bfgreek{rho} = \frac12 \sum_{\alpha \in \Phi^+} \alpha$ denote half the sum of all positive roots.

\medskip
\paragraph{\it Characters of the compact torus in $\wedge^\bullet \p_c$}

\begin{lemma}
\label{lem:characters-wedge-u}
Let $\gamma$ be a character of ${}^\circ T$ that appears in $\wedge^q \p^1_c$. Then there exist subsets $A$ and $B$ of  $\Phi^+$, with $|A| + |B| \leq q$, and written additively, 
$\gamma = \sum_{\alpha \in A} \alpha - \sum_{\beta \in B} \beta.$
If $r = q-|A|-|B|$, then this character is realized on any vector of the form 
$
(H_{i_1} \wedge \cdots \wedge H_{i_r}) \wedge (\bigwedge_{\alpha \in A} X_\alpha^p ) \wedge (\bigwedge_{\beta \in B} Y_\beta^p),
$
where $\{H_{i_1},\dots,H_{i_r}\}$ is a subset of a basis $\{H_1, \dots , H_{n-1}\}$ of ${}^\circ \a^1_c$. Furthermore, for the usual partial ordering on weights, one has
$- 2\bfgreek{rho} \leq \gamma \leq 2\bfgreek{rho}.$
\end{lemma}

The proof of the above lemma is an easy exercise that is left to the reader.

\medskip
\paragraph{\it Pure weight and finite-dimensional representation}
\label{par:pure-weight}
By a pure weight one means $\blambda = (\lambda, \lambda^*)$, where 
$\lambda = (b_1,\dots, b_n),$ 
$\lambda^* = (b_1^*,\dots, b_n^*),$ 
with $b_1 \geq \cdots \geq b_n$  and $b_1^* \geq \cdots \geq b_n^*$ are non-increasing tuples of integers satisfying the purity condition that there exists 
an integer $\w$ such that $b_j^* + b_{n-j+1} = \w.$ If $w_0$ denotes the Weyl group element of longest length in $G$, then the purity condition can also be written as 
\begin{equation}
\label{eqn:purity-lambda}
\lambda^* \ = \ -w_0(\lambda) + \w \cdot \bfgreek{delta}_n.
\end{equation}
(See \ref{sec:characters-of-torus}.) Recall that $(\rho_\lambda, \M_\lambda)$ is the irreducible 
representation of $G$ on the complex vector space $\M_\lambda$ 
with highest weight $\lambda$. It follows that $\M_{\lambda^*} = \M_\lambda^\v \otimes {\rm det}^\w,$ where $\M_\lambda^\v$ is the representation space of the contragredient 
$\rho_\lambda^\v$ of $\rho_\lambda.$ The highest weight of $(\rho_\lambda^\v, \M_\lambda^\v)$ is $-w_0(\lambda)$, i.e., $\rho_\lambda^\v \simeq  \rho_{w_0(\lambda)}.$ 
The automorphism $g \mapsto {}^t\! g^{-1}$ takes an irreducible representation of $\GL_n(\C)$ to its contragredient. These statements on the contragredient are 
expressed symbolically as: 
\begin{equation}
\label{eqn:contragredient-equivalences}
\rho_\lambda(\, {}^t\!g^{-1}) \ \simeq \ \rho_\lambda^\v(g) \ \simeq \ \rho_{-w_0(\lambda)}(g).
\end{equation}
The character $\blambda$ is an algebraic character of the torus $T$, thought of as a real group, that maps ${\rm diag}(z_1, z_2, \dots, z_n) \in T,$ with  $z_1,\dots, z_n \in \C^*,$ to $\prod_{j=1}^n z_j^{b_j} \bar{z}_j^{b_j^*};$ given $\blambda$ put 
$$
\M_{\blambda} = \M_\lambda \otimes \M_{\lambda^*},
$$
on which an element $g \in G$ acts via 
$\rho_{\blambda}(g) = \rho_\lambda(g) \otimes \rho_{\lambda^*}(\bar{g}).$ From \eqref{eqn:purity-lambda} and \eqref{eqn:contragredient-equivalences} we get 
\begin{equation}
\label{eqn:rho-blambda-g}
\rho_{\blambda}(g) \ = \ \rho_\lambda(g) \otimes \rho_\lambda(\theta(g)) \otimes \det(\bar{g})^\w.
\end{equation}
The $\theta$-conjugate of $(\rho_{\blambda}, \M_{\blambda})$ is the representation $({}^\theta\!\rho_{\blambda}, {}^\theta\!\M_{\blambda})$ defined as 
$$
{}^\theta\!\M_{\blambda} = \M_{\blambda}, \quad {}^\theta\!\rho_{\blambda}(g) = \rho_{\blambda}(\theta(g)).
$$
The restriction of the representation $(\rho_\lambda, \M_\lambda)$ of $G$ to $G^1$ is denoted $(\rho_\lambda^1, \M_\lambda^1)$, where of course $\M_\lambda^1 = \M_\lambda$ and 
$\rho_\lambda^1(g) = \rho_\lambda(g)$ for all $g \in G^1.$ Similarly, $\rho_{\blambda}^1(g) = \rho_{(\lambda, \lambda^*)}^1(g) = \rho_{\lambda}^1(g) \otimes \rho_{\lambda^*}^1(\bar{g})$ for all 
$g \in G^1.$  From \eqref{eqn:rho-blambda-g} it follows that ${}^\theta\!\rho_{(\lambda, \lambda^*)}^1 = \rho_{(\lambda^*, \lambda)}^1$ from which one has: 

\begin{lemma}
\label{lem:I_0}
There exists an isomorphism $I_0 : \M_{\blambda}^1 \to \M_{\blambda}^1$ with the intertwining property: 
$$
I_0 \circ \rho_{\blambda}^1(g) \ = \ \rho_{\blambda}^1(\theta(g)) \circ I_0, \quad \forall g \in G^1.
$$
\end{lemma}

\medskip
\paragraph{\it A unitary representation of $G^1$} 
\label{par:unitary-rep}
 For a pure weight $\blambda = (\lambda, \lambda^*)$ as in \ref{par:pure-weight}, define its cuspidal parameters as: 
 $\alpha = -w_0(\lambda) + \bfgreek{rho}$ and $\beta = -\lambda^* - \bfgreek{rho}.$ Writing $\alpha = (\alpha_1,\dots,\alpha_n)$ and 
 $\beta = (\beta_1,\dots,\beta_n)$, one has: 
 $ \alpha_j = -b_{n-j+1} + (n-2j+1)/2$ and $\beta_j = -b_j^* - (n-2j+1)/2.$
 Note that $\alpha_j + \beta_j = -\w$. Consider the character $\chi_{\blambda} : T \to \C^*$ given by:
$$
\chi_{\blambda}({\rm diag}(z_1, z_2, \dots, z_n)) \ = \ \prod_{j=1}^n z_j^{\alpha_j} \bar{z}_j^{\beta_j}.
$$
Define the parabolically induced representation of $G = \GL_n(\C)$ via normalized induction: 
$$
(\psi_{\blambda}, \J_{\blambda}) \ = \ \Ind_B^G(\chi_{\blambda}).
$$
The restriction of the representation $(\psi_{\blambda}, \J_{\blambda})$ to $G^1 = \SL_n(\C)$ is denoted $(\psi_{\blambda}^1, \J_{\blambda}^1)$; since $G = B\cdot G^1$, we have 
 $$
(\psi_{\blambda}^1, \J_{\blambda}^1) \ = \ \Ind_{B^1}^{G^1}(\chi_{\blambda}^1), 
$$
where  $\chi_{\blambda}^1$ is the restriction of $\chi_{\blambda}$ to $T^1$. If ${\rm diag}(z_1, z_2, \dots, z_n) \in T^1$ then $z_n = (z_1 z_2 \cdots z_{n-1})^{-1}$, from which it follows
that $\chi_{\blambda}^1$ is the unitary character of $T^1$ given by:  
\begin{equation}
\label{eqn:unitary-character}
\chi_{\blambda}^1({\rm diag}(z_1, z_2, \dots, z_n)) \ = \
\prod_{j=1}^{n-1} \left(\frac{z_j}{\bar{z}_j}\right)^{b_1 - b_{n-j+1} + n - j}.
\end{equation}
Hence, $(\psi_{\blambda}^1, \J_{\blambda}^1)$ is an irreducible unitary representation of $G^1.$

\begin{lemma}
\label{lem:I_infty}
There exists an isomorphism $I_\infty : \Ind_{B^1}^{G^1}(\chi_{\blambda}^1) \to \Ind_{B^1}^{G^1}(\chi_{\blambda}^1)$ with the intertwining property: 
$$
I_\infty \circ \psi_{\blambda}^1(g) \ = \ \psi_{\blambda}^1(\theta(g)) \circ I_\infty, \quad \forall g \in G^1.
$$
\end{lemma}

This follows from the theory of standard intertwining operators. For a function $f$ in the induced space $\Ind_{B^1}^{G^1}(\chi_{\blambda}^1)$, the integral 
$$
I_\infty(f)(g) \ = \ \int_N f(w_0 n w_0^{-1} \theta(g))\, dn, \quad \mbox{($g \in G^1$)},
$$
after an interpretation via analytic continuation, does the job. Another way to organize one's thoughts is to consider the operator $f \mapsto {}^\theta\! f$, where 
${}^\theta\! f(g) = f(\theta(g))$ for 
$f \in \Ind_{B^1}^{G^1}(\chi_{\blambda}^1),$ 
which has the correct intertwining property except
${}^\theta f$ lives in the induced space where we induce from the opposing Borel subgroup of all lower triangular matrices; then 
appeal to the fact that a parabolically induced representation depends only on the representation of the Levi subgroup and is independent of the choice of parabolic subgroup with that Levi--a fact 
whose proof uses the above intertwining integral. 

\medskip
\paragraph{\it Cohomology of $\J^1_{\blambda} \otimes \M^1_{\blambda}$}

\begin{lemma}
\label{lem:coh-j-lambda}
Let $\blambda$ be a pure weight as above. We have
$$
H^q(\g^1, K^1; \, \J^1_{\blambda} \otimes \M^1_{\blambda}) \ = \ \Hom_{K^1}(\wedge^q \p^1_c, \, \J^1_{\blambda} \otimes \M^1_{\blambda}).
$$
Furthermore, 
$$
\dim_\C(H^q(\g^1, K^1; \J^1_{\blambda} \otimes \M^1_{\blambda})) \ = \ 
\binom{n-1}{q - n(n-1)/2}. 
$$
\end{lemma}
The first assertion is standard; see, for example, Wallach \cite[Prop.\,9.4.3]{wallach}. 
The statement about dimensions is contained in Clozel \cite[Lemme 3.14]{clozel}.

\medskip
\paragraph{\it The Lefschetz number to compute}
Using the intertwining operators in Lem.\,\ref{lem:I_0} and Lem.\,\ref{lem:I_infty}, define 
\begin{equation}
\label{eqn:int-op}
I :  \J^1_{\blambda} \otimes \M^1_{\blambda} \to \J^1_{\blambda} \otimes \M^1_{\blambda}, \quad I := I_\infty \otimes I_0.
\end{equation}
The automorphism $\theta$ on $G$ induces the automorphism $\theta^q$ of $H^q(\g^1, K^1; \, \J^1_{\blambda} \otimes \M^1_{\blambda})$, 
which can be described, after appealing to Lem.\,\ref{lem:coh-j-lambda}, as mapping
 $f \in \Hom_{K^1}(\wedge^q \p^1_c, \, \J^1_{\blambda} \otimes \M^1_{\blambda})$ to 
$\theta^q(f)$ where $\theta^q(f)(X) = I(f(\theta(X)))$ for all $X \in \wedge^q \p^1_c.$ The rest of this subsection is devoted to computing the Lefschetz number
$$
\Lef(\theta, \SL_n(\C), \J^1_{\blambda} \otimes \M^1_{\blambda}) = \sum_q (-1)^q \, {\rm Trace}(\theta^q),
$$ 
and specifically to showing that it is nonzero. See Prop.\,\ref{prop:slnc-lefschetz} below.

\medskip
\paragraph{\it Relevant $K^1$-types of $\J^1_{\blambda} \otimes \M^1_{\blambda}$}

After Lem.\,\ref{lem:coh-j-lambda}, we need to understand the $K^1$-types that occur both in $\wedge^q \p^1_c$ and in 
$$
\J^1_{\blambda} \otimes \M^1_{\blambda} = \Ind_{B^1}^{G^1}(\chi^1_{\blambda}) \otimes \M^1_{\blambda} = \Ind_{B^1}^{G^1}(\chi^1_{\blambda} \otimes \M^1_{\blambda}|_{B^1}).
$$
For the latter, using $G^1 = B^1K^1$ and that $B^1 \cap K^1 = {}^\circ T^1$, one has
$$
\Res_{K^1}(\J^1_{\blambda} \otimes \M^1_{\blambda}) \ = \ 
\Res_{K^1}(\Ind_{B^1}^{G^1}(\chi^1_{\blambda} \otimes \M^1_{\blambda}|_{B^1})) \ \simeq \ 
\Ind_{{}^\circ T^1}^{K^1}(\chi^1_{\blambda} \otimes \M^1_{\blambda}|_{{}^\circ T^1}).
$$
For $\wedge^q \p^1_c$, using $\p^1_c = \a^1_c \oplus \u^1_c$, and that the adjoint action of ${}^\circ T^1$ on $\a^1_c$ is trivial, we get
\begin{equation}
\label{eqn:kunneth-frobenius}
\begin{split}
\Hom_{K^1}(\wedge^q \p^1_c, \, \J^1_{\blambda} \otimes \M^1_{\blambda}) 
& \simeq \Hom_{K^1}(\wedge^q \p^1_c, \, \Ind_{{}^\circ T^1}^{K^1}(\chi^1_{\blambda} \otimes \M^1_{\blambda}|_{{}^\circ T^1}))  \\
& \simeq  \Hom_{{}^\circ T^1}(\wedge^q \p^1_c, \, \chi^1_{\blambda} \otimes \M^1_{\blambda}|_{{}^\circ T^1}) \\ 
 & \simeq \bigoplus_{s+t = q} \Hom_{{}^\circ T^1}(\wedge^s \a^1_c \otimes \wedge^t \u_c, \, \chi^1_{\blambda} \otimes \M^1_{\blambda}|_{{}^\circ T^1}) \\ 
 & \simeq  \bigoplus_{s+t = q} \wedge^s \a_c^{1*} \otimes  \Hom_{{}^\circ T^1}(\wedge^t \u_c, \, \chi^1_{\blambda} \otimes \M^1_{\blambda}|_{{}^\circ T^1}). 
\end{split}
\end{equation}

 \medskip
\paragraph{\it Relevant ${}^\circ T^1$-characters in $\chi^1_{\blambda} \otimes \M^1_{\blambda}|_{{}^\circ T}$}
From \eqref{eqn:kunneth-frobenius} and the statement about dimension in Lem.\,\ref{lem:coh-j-lambda} we deduce the following:  

\begin{lemma}
\label{lem:canonical-character}
$$
\Hom_{{}^\circ T^1}(\wedge^t \u_c, \chi^1_{\blambda} \otimes \M^1_{\blambda}|_{{}^\circ T^1}) \ = \ 
\begin{cases}
\C & \mbox{if $t = \frac{n(n-1)}{2}$,} \\
0 & \mbox{otherwise}, 
\end{cases}
$$
i.e., there is a unique character of the compact torus ${}^\circ T^1$ denoted, say, $\chi_0$, that appears in $\wedge^\bullet \u_c$ and in $\chi^1_{\blambda} \otimes \M^1_{\blambda}|_{{}^\circ T^1}$;  necessarily, it appears in a unique degree; this degree is $\bullet = n(n-1)/2 = \dim(\u_c)/2$; 
furthermore, it appears with multiplicity one in $\wedge^{n(n-1)/2}\u_c$ and in 
$\chi^1_{\blambda} \otimes \M^1_{\blambda}|_{{}^\circ T^1}$. 
\end{lemma}

 \medskip
\paragraph{\it The canonical character $\chi_0$ of ${}^\circ T^1$}

\begin{lemma}
\label{lem:canonical-chi_0}
The canonical character $\chi_0$ of Lem.\,\ref{lem:canonical-character} is given by:
$$
\chi_0 \ = \ 2 \bfgreek{rho} \ = \ \sum_{\alpha \in \Phi^+} \alpha.
$$
\end{lemma}

\begin{proof}
From Lem.\,\ref{lem:characters-wedge-u} one sees that the character $2 \bfgreek{rho}$ is realized by the weight vector $\wedge_{\alpha \in \Phi^+} X_\alpha^p$ 
in $\wedge^{n(n-1)/2}\u_c$. Recall that $2\bfgreek{rho}$ as a character of $T$ is represented by $(n-1,n-3,\dots, 1-n)$, and as a character of $T^1$ it is the character which on 
$\underline{z} := {\rm diag}(z_1,\dots, z_n) \in T^1$ (note: $z_j \in \C^*$, and $z_n = (z_1z_2\cdots z_{n-1})^{-1}$) is: 
$2\bfgreek{rho}(\underline{z}) \ = \ \prod_{j=1}^{n-1} z_j^{2(n-j)}.$
Next, recall $\blambda = (\lambda, \lambda^*)$ and related notations from \ref{par:pure-weight}; let $v_\lambda \in \M_\lambda$ be 
the weight vector realizing the highest weight $\lambda$; similarly, $v_{\lambda^*} \in \M_{\lambda^*}$; 
for the Weyl group element $w_0$, the character $(w_0\lambda, \lambda^*)$ is an extremal weight, realized by 
the action of 
${}^\circ T$ on the weight vector $\rho_{\lambda}(w_0) v_\lambda \otimes v_{\lambda^*} \in \M_{\blambda}$ which is described by 
$
\underline{z} = {\rm diag}(z_1, z_2, \dots, z_n) \mapsto \prod_{j=1}^n z_j^{b_{n-j+1}} \cdot \bar{z}_j^{b_j^*}.$
If $\underline{z} \in {}^\circ T^1$, i.e., $\bar{z_j} = z_j^{-1}$ and $z_n = (z_1z_2\cdots z_{n-1})^{-1}$, then one sees that ${}^\circ T^1$ acts on 
$\rho_{\lambda}(w_0) v_\lambda \otimes v_{\lambda^*}$ by the character:
$$
\underline{z} \mapsto \prod_{j=1}^{n-1} z_j^{(b_{n-j+1} - b_1) - (b_j^*-b_n^*)} \ = \ \prod_{j=1}^{n-1} z_j^{2(b_{n-j+1} - b_1)}; 
$$
hence, from \eqref{eqn:unitary-character}, the vector $\rho_{\lambda}(w_0) v_\lambda \otimes v_{\lambda^*}$ is a weight vector in 
$\chi^1_{\blambda} \otimes \M^1_{\blambda}|_{{}^\circ T^1}$ for the character 
$$
\underline{z} \mapsto 
\prod_{j=1}^{n-1} \left( \frac{z_j}{\bar{z}_j} \right)^{b_1 - b_{n-j+1} + n - j} \cdot 
\prod_{j=1}^{n-1} z_j^{2(b_{n-j+1} - b_1)} \ = \ 
\prod_{j=1}^{n-1} z_j^{2(n-j)} \ = \ (2\bfgreek{rho})(\underline{z});
$$
whence, by uniqueness part of Lem.\,\ref{lem:canonical-character}, $\chi_0 = 2\bfgreek{rho}.$
\end{proof}

\medskip
\paragraph{\it The canonical $K$-type supporting cohomology}
For any compact Lie group $\K$ (such as $K^1$ or ${}^\circ T^1$), and a $\K$-module $V,$ 
the set of isomorphism classes of irreducible representations of $\K$ that appear in $V$ is denoted $\Spec_\K(V).$ We begin by noting that a $K^1$-type that supports cohomology 
contains the canonical character $\chi_0.$

\begin{lemma}
\label{lem:K-type-chi-0}
Let $\pi$ be an irreducible representation of $K^1$ that appears in $\wedge^q \p_c^1$  and also in $\Ind_{{}^\circ T^1}^{K^1}(\chi^1_{\blambda} \otimes \M^1_{\blambda}|_{{}^\circ T^1})$.
Then:
\smallskip
\begin{enumerate}
\item[(i)] $\Spec_{{}^\circ T^1}(\pi|_{{}^\circ T^1}) \, \cap \, \Spec_{{}^\circ T^1}(\chi^1_{\blambda} \otimes \M^1_{\blambda}|_{{}^\circ T^1}) \ = \ \{\chi_0\}.$ 

\smallskip
\item[(ii)] $\Hom_{K^1} \left(\pi, \, \Ind_{{}^\circ T^1}^{K^1}(\chi_0) \right) \ = \ 
\Hom_{K^1}\left(\pi, \Ind_{{}^\circ T^1}^{K^1}(\chi^1_{\blambda} \otimes \M^1_{\blambda}|_{{}^\circ T^1}) \right).$
\end{enumerate}
\end{lemma}

\begin{proof}
Since $\pi$ is irreducible, appearing in $\wedge^q \p^1_c$ means that it embeds, from which (i) follows by using Lem.\,\ref{lem:canonical-character} and Frobenius reciprocity. 
For (ii), observe that both the $\Hom$-spaces are equal to $\Hom_{{}^\circ T^1}(\pi|_{{}^\circ T^1}, \chi_0).$
\end{proof}

\begin{lemma}
\label{lem:K-type-inducing-chi-0}
$$
\Hom_{K^1} \left(\wedge^q \p^1_c, \, \Ind_{{}^\circ T^1}^{K^1}(\chi^1_{\blambda} \otimes \M^1_{\blambda}|_{{}^\circ T^1}) \right) \ = \ 
\Hom_{K^1} \left(\wedge^q \p^1_c, \, \Ind_{{}^\circ T^1}^{K^1}(\chi_0)\right).
$$
\end{lemma}

\begin{proof}
Using Frobenius reciprocity and Lem.\,\ref{lem:canonical-character} one has:
\begin{equation*}
\begin{split}
\Hom_{K^1}(\wedge^q \p^1_c, \, \Ind_{{}^\circ T^1}^{K^1}(\chi^1_{\blambda} \otimes \M^1_{\blambda}|_{{}^\circ T^1}))
 & \ = \ \Hom_{{}^\circ T^1}(\wedge^q \p^1_c, \, \chi^1_{\blambda} \otimes \M^1_{\blambda}|_{{}^\circ T^1}) \\
 & \ = \ \Hom_{{}^\circ T^1}(\wedge^q \p^1_c, \, \chi_0) \otimes \Hom_{{}^\circ T^1}(\chi_0,  \chi^1_{\blambda} \otimes \M^1_{\blambda}|_{{}^\circ T^1}).
\end{split}
\end{equation*}
Since $\chi_0$ is a weight of $\chi^1_{\blambda} \otimes \M^1_{\blambda}$ that came from an extremal weight of $\M^1_{\blambda}$ as in the proof of Lem.\,\ref{lem:canonical-chi_0}, it appears with multiplicity one, i.e., the space  
$\Hom_{{}^\circ T^1}(\chi_0,  \chi^1_{\blambda} \otimes \M_{\blambda}|_{{}^\circ T})$ is one-dimensional. The lemma follows using Frobenius reciprocity again. 
\end{proof}

For a character $\chi$ of ${}^\circ T^1$, which is a dominant integral weight for the Lie group $K^1$, let $\pi_\chi$ denote the irreducible representation of $K^1$ with 
highest weight $\chi.$ For the canonical character $\chi_0$ of Lem.\,\ref{lem:canonical-chi_0}, the irreducible representation $\pi_{\chi_0}$ is the unique $K$-type that supports 
cohomology--this is made precise in the following lemma. This is inspired by the work of Vogan and Zuckerman; especially, see \cite[Cor.\,3.7]{vogan-zuckerman}. 

\begin{lemma}
\label{lem:canonical-K-type-coh}
The irreducible representation $\pi_{\chi_0}$ of $K^1$ with highest weight $\chi_0$ (Lem.\,\ref{lem:canonical-chi_0}) 
appears with multiplicity one in $\Ind_{{}^\circ T^1}^{K^1}(\chi^1_{\blambda} \otimes \M^1_{\blambda}|_{{}^\circ T^1})$, and this inclusion 
induces the equality: 
$$
\Hom_{K^1} \left(\wedge^q \p^1_c, \, \pi_{\chi_0} \right) \ = \ 
\Hom_{K^1} \left(\wedge^q \p^1_c, \, \Ind_{{}^\circ T^1}^{K^1}(\chi^1_{\blambda} \otimes \M^1_{\blambda}|_{{}^\circ T^1})\right). 
$$
\end{lemma}

\begin{proof}
If $\pi_\chi$, the irreducible representation of $K^1$ with 
highest weight $\chi,$ appears in $\wedge^q \p^1_c$ then from Lem.\,\ref{lem:characters-wedge-u} one has: 
\begin{equation}
\label{eqn:restriction-chi}
- 2 \bfgreek{rho} \ \leq \ \chi \ \leq \ 2\bfgreek{rho}. 
\end{equation}
(Recall: $\chi_0 = 2 \bfgreek{rho}.$) 
Next, for any unitary representation $\Pi$ of $K^1$ one has the Hilbert-space decomposition into its isotypic components: 
$\Pi \ = \ \bigoplus_{\chi \in \Lambda} \pi_\chi \otimes \Hom_{K^1}(\pi_\chi, \Pi),$
where $\chi$ runs over the set $\Lambda$ of all dominant integral weights. Using Lem.\,\ref{lem:K-type-inducing-chi-0} one deduces
$$
\Hom_{K^1}(\wedge^q \p^1_c, \, \Ind_{{}^\circ T^1}^{K^1}(\chi^1_{\blambda} \otimes \M^1_{\blambda}|_{{}^\circ T^1})) \ = \ 
\Hom_{K^1}(\wedge^q \p^1_c, \, \bigoplus_{\chi \in \Lambda} \pi^1_\chi \otimes \Hom_{K^1}(\pi^1_\chi, \Ind_{{}^\circ T^1}^{K^1}(\chi_0)).
$$
From \eqref{eqn:restriction-chi} and Frobenius reciprocity, the space on the right hand side is equal to
$$
\bigoplus_{- 2\bfgreek{rho} \leq  \chi \leq  2\bfgreek{rho}} 
\Hom_{K^1}(\wedge^q \p^1_c, \, \pi^1_\chi) \otimes \Hom_{{}^\circ T^1}(\pi^1_\chi, \chi_0).
$$
If $\Hom_{{}^\circ T^1}(\pi^1_\chi, \chi_0) \neq 0$ then $\chi_0 \leq \chi$, hence only the summand for $\chi = \chi_0$ survives, and furthermore $\Hom_{{}^\circ T^1}(\pi^1_{\chi_0}, \chi_0)$ is one-dimensional by usual highest-weight theory; whence  
$$
\Hom_{K^1}(\wedge^q \p^1_c, \, \Ind_{{}^\circ T^1}^{K^1}(\chi^1_{\blambda} \otimes \M^1_{\blambda}|_{{}^\circ T^1})) = \Hom_{K^1}(\wedge^q \p^1_c, \, \pi_{\chi_0}).
$$
This also shows that $\pi_{\chi_0}$ appears with multiplicity one in $\Ind_{{}^\circ T^1}^{K^1}(\chi^1_{\blambda} \otimes \M^1_{\blambda}|_{{}^\circ T^1})$ as the above equality means that 
any $K^1$-equivariant map from $\wedge^q \p^1_c$ to $\Ind_{{}^\circ T^1}^{K^1}(\chi^1_{\blambda} \otimes \M^1_{\blambda}|_{{}^\circ T^1})$ factors as 
$\wedge^q \p^1_c \ \twoheadrightarrow \ \pi_{\chi_0} \ \hookrightarrow \ \Ind_{{}^\circ T^1}^{K^1}(\chi^1_{\blambda} \otimes \M^1_{\blambda}|_{{}^\circ T^1}).$
\end{proof}

\medskip
\paragraph{\it The main result on Lefschetz numbers for $\SL_n(\C)$}

Suppose $V_{\pi_{\chi_0}}$ is the subspace of the induced space $\Ind_{{}^\circ T^1}^{K^1}(\chi^1_{\blambda} \otimes \M^1_{\blambda}|_{{}^\circ T^1})$ realizing the unique 
occurrence of $\pi_{\chi_0}$, then 
the intertwining operator $I$ from \eqref{eqn:int-op} has to stabilize $V_{\pi_{\chi_0}}$. By Schur's lemma, $I$ acts on $V_{\pi_{\chi_0}}$ as homothety by a 
scalar $c_{\chi_0}$, say. Since $I$ is an isomorphism, $c_{\chi_0} \neq 0.$

\medskip
\begin{prop}
\label{prop:slnc-lefschetz}
$$
\Lef(\theta, \SL_n(\C), \J^1_{\blambda} \otimes \M^1_{\blambda}) \ = \ c_{\chi_0} 2^{n-1}  \ \neq \ 0.
$$
\end{prop}

\begin{proof}
Recall the map $\theta^q$ of $H^q(\g^1, K^1; \, \J^1_{\blambda} \otimes \M^1_{\blambda})$ which on $f \in \Hom_{K^1}(\wedge^q \p^1_c, \, \J^1_{\blambda} \otimes \M^1_{\blambda})$ gives  
$\theta^q(f)$ where $\theta^q(f)(X) = I(f(\theta(X))$ for all $X \in \wedge^q \p^1_c.$ Now, using Lem.\,\ref{lem:canonical-K-type-coh} and that $\theta$ acts as $-1$ on $\p^1_c$ 
one deduces that $\theta^q$ acts on $\Hom_{K^1}(\wedge^q \p^1_c, \, \pi_{\chi_0})$ as a homothety by $(-1)^q c_{\chi_0}$. Lem.\,\ref{lem:coh-j-lambda} implies that
$
{\rm Trace}(\theta^q) \ = \ (-1)^q c_{\chi_0} \binom{n-1}{q - n(n-1)/2}.$
Hence $\Lef(\theta, \SL_n(\C), \J^1_{\blambda} \otimes \M^1_{\blambda}) = \sum_q (-1)^q {\rm Trace}(\theta^q) = c_{\chi_0} \sum_q \binom{n-1}{q - n(n-1)/2} = c_{\chi_0}2^{n-1}.$
\end{proof}

\medskip
\subsubsection{Lefschetz number for $\SL_n(\R)$}
\label{sec:slnr}

We will now prove the analogue of Prop.\,\ref{prop:slnc-lefschetz} for $\SL_n(\R)$. The proofs are very similar after making appropriate changes and so we will be brief.

\medskip
\paragraph{\it Notations}
\label{par:notations-slnr}

For all of \ref{sec:slnr}, let $G= \GL_n(\R)$, and $G^1 = \SL_n(\R)$; let $K = \rO(n)$ (resp., $K^1 = \SO(n)$) be the maximal compact subgroup of $G$ (resp., $G^1).$ Let $\theta$ be the involution on $G$ defined as 
$\theta(g)={}^tg^{-1}$ for $g \in G.$ The set of fixed points of $\theta$ in $G^1$ is $K^1.$ Let $\g = \gl_n(\R)$ (resp., $\g^1 = \sl_n(\R)$) be the Lie algebra of the Lie group $G$ (resp., $G^1$) on which the action of $\theta$, which we continue to denote by $\theta$, is $\theta(X) = - {}^t \! X$ for $X \in \g.$ On $\g^1$, the $+1$ eigenspace of $\theta$ is $\k^1$, the Lie algebra of $K^1,$ 
consisting of all skew-symmetric matrices in $\g$ or $\g^1$, and the $-1$ eigenspace of $\theta$ is denoted $\p^1$ which consists of all symmetric matrices in $\g^1$; one has $\g^1 = \k^1 \oplus \p^1,$ a decomposition of real Lie algebras.  Denote by $\g^1_c$, $\k^1_c$, and $\p^1_c$ the complexifications of $\g^1$, $\k^1,$ and $\p^1,$ respectively. Identify $\g^1_c/\k^1_c$ with $\p^1_c$.
Let $T=\{{\rm diag} (x_1,\ldots,x_n)\in G: x_i \in \R^*\}$ be the diagonal torus in $G$; let $\Phi$ be the root system identified with tuples $1\leq i\neq j\leq n$ in the usual way, with $\Phi^+$ the positive roots corresponding to $i < j$. Let $\cW$ be the Weyl group $N_G(T)/T$.

\medskip
\paragraph{\it Pure weight and finite-dimensional representations} 

Let $\lambda=(b_1,\ldots,b_n) \in \Z^n$ with $b_1 \geq b_2 \geq \dots \geq b_n$ be a dominant integral weight for $G$. Assume that $\lambda$ satisfies the purity condition 
$b_i + b_{n-i+1} = \w$ for all $1 \leq i \leq n.$ 
Let $(\rho_{\lambda},\mathcal{M}_{\lambda})$ be the finite-dimensional irreducible (complex-)representation of $G$ with the highest weight $\lambda$; its restriction to $G^1$, 
denoted $(\rho^1_{\lambda},\mathcal{M}^1_{\lambda})$ is irreducible and has highest weight $\lambda^1$, i.e., 
$(\rho_{\lambda^1},\mathcal{M}_{\lambda^1}) = (\rho^1_{\lambda},\mathcal{M}^1_{\lambda})$.

\begin{lemma}
\label{lem:I_0-SLnR}
 There exists an isomorphism $I_0 : \mathcal{M}^1_{\lambda} \to \mathcal{M}^1_{\lambda}$ with the intertwining property: 
    $$ 
    I_0 \circ \rho^1_{\lambda}(g) = \rho^1_{\lambda}(\theta(g)) \circ I_0, \quad \forall g \in \SL_n(\mathbb{R}).
    $$
\end{lemma}

Follows from purity as in the discussion preceding Lem.\,\ref{lem:I_0}; the details are left to the reader.

\medskip
\paragraph{\it Unitary representation of $G^1$} 

For any integer $l\geq 1$, let $D_l$ be the discrete series representation of $\GL_2(\mathbb{R})$ with lowest non-negaive $\SO(2)$ type $e^{\iota \theta} \mapsto e^{\iota (l+1)\theta}$ and 
central character trivial on the connected component of the center of $\GL_2(\mathbb{R})$. The restriction of $D_l$ to $\SO(2)$ is 
$\oplus_k V_k,$ with the sum running over all integers $k,$ with $k \equiv l+1 \pmod{2}$ and  
$|k|\geq l+1$, where $V_k$ is the one-dimensional $\SO(2)$-type given by the character $e^{\iota \theta} \mapsto e^{\iota k\theta}$.
Suppose $\SL_q^\pm(\R)$ denotes the subgroup of $\GL_q(\R)$ consisting of all 
$q \times q$ matrices with determinant $\pm 1,$ then the restriction of $D_l$ to $\SL_2^\pm(\R)$ is irreducible. 
Define the cuspidal parameters $(l_1,l_2,\dots, l_n)$ for the pure weight $\lambda$ as $l_i =b_i-b_{n-i+1}+(n-2i+1)$. 
Let $P_{\ul n}$ be the standard parabolic subgroup of $G$ of all upper-triangular matrices corresponding to the partition 
$\ul{n}$ of $n$ given by $2 + \dots + 2$ or $n = 2 + \dots + 2 + 1$, depending on whether $n$ is even or odd.
Let $M_{\ul{n}} = {}^\circ M_{\ul{n}} \cdot A_{\ul{n}}$ denote the Langlands decomposition of its standard Levi subgroup $M_{\ul{n}}$; the connected component of the center of $M_{\ul{n}}$ is $A_{\ul{n}}$. Then 
$M_{\ul{n}} = \GL_2(\R) \times \dots \times \GL_2(\R) (\times \GL_1(\R))$, the notation means that the last $\GL_1(\R)$ factor is present only when $n$ is odd, and 
$$
{}^\circ M_{\ul{n}} = \SL_2^{\pm}(\R) \times \dots \times \SL_2^{\pm}(\R) (\times \SL_1^{\pm}(\R)).
$$
Put ${}^\circ D_\lambda = D_{l_1} \otimes \dots \otimes D_{l_{[n/2]}} (\otimes 1\!\!1)$ as a representation of ${}^\circ M_{\ul{n}}$, where the last factor $1\!\!1$ is the trivial character of 
$\SL_1^{\pm}(\R) =  \{\pm 1\}$ which is present when $n$ is odd. 
Let $M_{\ul{n}}^1 = M_{\ul{n}} \cap G^1$, then 
\begin{equation}
\label{eqn:M-n-1}
M_{\ul{n}}^1 \simeq 
\begin{cases} 
\GL_2(\R) \times \dots \times \GL_2(\R) & \mbox{if $n$ is odd,} \\
{\rm S}(\GL_2(\R) \times \dots \times \GL_2(\R)) & \mbox{if $n$ is even.} 
\end{cases}
\end{equation}
For the parabolic subgroup $P_{\ul{n}}^1 = P_{\ul{n}} \cap G^1$ of $G^1$, then 
$M_{\ul{n}}^1 = {}^\circ M_{\ul{n}}^1 \cdot A_{\ul{n}}^1$ is the Langlands decomposition of its standard Levi subgroup, where 
$A_{\ul{n}}^1 = A_{\ul{n}} \cap G^1$ and ${}^\circ M_{\ul{n}}^1 = {}^\circ M_{\ul{n}} \cap G^1$. Observe that 
$$
{}^\circ M_{\ul{n}}^1 \simeq 
\begin{cases} 
\SL_2^{\pm}(\R) \times \dots \times \SL_2^{\pm}(\R) & \mbox{if $n$ is odd,} \\
{\rm S}(\SL_2^{\pm}(\R) \times \dots \times \SL_2^{\pm}(\R)) & \mbox{if $n$ is even.} 
\end{cases}
$$
Let ${}^\circ D_\lambda^1$ be the restriction of ${}^\circ D_\lambda$ to ${}^\circ M_{\ul{n}}^1$; this is irreducible if $n$ is odd, and 
is a sum of two inequivalent irreducible representations if $n$ is even. To homogenize notation with \ref{sec:slnc}, let $\chi_\lambda^1$ denote an irreducible summand of 
${}^\circ D_\lambda^1$. 
 After extending $\chi_\lambda^1$ trivially across 
$A_{\ul{n}}^1$, consider the parabolically induced representation:
$$
(\psi_{\lambda}^1, \J_{\lambda}^1)\ := \ \operatorname{Ind}_{P_{\ul{n}}^1}^{G^1} (\chi_\lambda^1).
$$ 

\begin{lemma}
\label{lem:I_infty-slnr}
    There exists an isomorphism 
    $I_\infty : \operatorname{Ind}_{P_{\ul{n}}^1}^{G^1} (\chi_\lambda^1) \rightarrow \operatorname{Ind}_{P_{\ul{n}}^1}^{G^1} (\chi_\lambda^1)$ 
with the intertwining property: $I_\infty \circ \psi_{\lambda}(g) = \psi_{\lambda}(\theta(g)) \circ I_\infty$ for all $g \in G^1.$
\end{lemma}

This is as in Lem.\,\ref{lem:I_infty}, together with the observation that when $n$ is even, $\theta$ stabilizes each irreducible summand of 
the restriction of ${}^\circ D_\lambda$ to ${}^\circ M_{\ul{n}}^1$, in particular, it stabilizes $\chi_\lambda^1$.

\medskip
\paragraph{\it Cohomology of $\J_\lambda^1 \otimes \M_\lambda^1$}

The following lemma is contained in Clozel \cite[Lemme 3.14]{clozel}.

\begin{lemma}
\label{lem:dimension-slnr}
Let $b_n = m^2$ if $n=2m$, and $b_n = m(m+1)$ if $n=2m+1$. Then 
 $$
 H^q(\mathfrak{g}^1,K^1, \J_\lambda^1 \otimes \M_\lambda^1) \ \cong \ 
 \wedge^{(q- b_n)} \mathfrak{a}_{\ul{n},c}^{1*} \ \cong \ 
 \begin{cases} 
 	\wedge^{q-m^2} (\mathbb{C}^{m-1} ) & n=2m, \\ 
	 \wedge^{q-m(m+1)} (\mathbb{C}^{m}) & n=2m+1. \end{cases}
 $$ 
\end{lemma}
It is helpful to note $b_n = \lfloor n^2/4 \rfloor$ and $\dim(\a_{\ul{n},c}^1) = \lfloor (n-1)/2 \rfloor.$

\medskip
\paragraph{\it Lefschetz number to compute}

The operator $I=I_{\infty}\otimes I_0$ defines an isomorphism on $\J_{\lambda}^1 \otimes \M_{\lambda}^1$ such that 
$I \circ \big(\psi_{\lambda}^1 \otimes\rho_{\lambda}^1 \big)(g) = \big(\psi_{\lambda}^1 \otimes\rho_{\lambda}^1\big)(\theta(g)) \circ I,$ for all $g \in G^1.$
Since $\theta(k)=k$ for all $k\in K^1$, the operator $I$ is a self-intertwining map of $\psi_{\lambda}^1 \otimes\rho_{\lambda}^1$ as a $K^1$-module. 
The map $\theta^q$ induced by $\theta$ on the cohomology group 
$H^q(\mathfrak{g}^1,K^1,\J_{\lambda}^1 \otimes \M_{\lambda}^1)\cong \operatorname{Hom}_{K^1}(\wedge^q \g^1/\k^1, \J_{\lambda}^1 \otimes \M_{\lambda}^1)$ is given by $f\mapsto \big(\theta^qf\big)(X):=I[f\big(\theta(X)\big)]$ for all $X\in \wedge^q \g^1/\k^1.$  
To compute $\operatorname{Lef}(\theta, \SL_n(\mathbb{R}),\J_{\lambda}^1 \otimes \M_{\lambda}^1)$, we need to compute 
$\operatorname{Trace}(\theta^q)$.

\medskip
\paragraph{\it Relevant $K^1$-types of $\J_\lambda^1 \otimes \M_\lambda^1$}

By \cite[Prop.\,9.4.3]{wallach}, the differential computing relative Lie algebra cohomology in the context at hand is zero, i.e., 
$$
H^q(\mathfrak{g}^1, K^1, \, \J_{\lambda}^1 \otimes \M_{\lambda}^1) \cong 
\Hom_{K^1}(\wedge^q(\g^1/\k^1), \, \J_{\lambda}^1 \otimes \M_{\lambda}^1).
$$
The cohomology depends only on $K^1$-types appearing in $\wedge^q \p^1$ and $\J_{\lambda}^1 \otimes \M_{\lambda}^1$. The restriction to $K^1$ of  
$\J_{\lambda}^1 \otimes \M_{\lambda}^1$ is identified with $\Ind_{K^1 \cap P_{\ul{n}}^1}^{K^1}(\chi_{\lambda}^1 \otimes \M_{\lambda}^1).$ 
By Frobenius reciprocity: 
$$
H^q(\mathfrak{g}^1, K^1, \, \J_{\lambda}^1 \otimes \M_{\lambda}^1) \ \cong \ 
\Hom_{K^1}(\wedge^q\mathfrak{p}^1_c, \, \Ind_{K^1\cap P_{\ul{n}}^1}^{K^1}(\chi_{\lambda}^1 \otimes \M_{\lambda}^1)) \ \cong \ 
\Hom_{K^1 \cap P_{\ul{n}}^1} (\wedge^q\mathfrak{p}^1_c, \, \chi_{\lambda}^1 \otimes \M_{\lambda}^1).
$$
Let $K_{\ul{n}}^1 := K^1 \cap P_{\ul{n}}^1 = K^1 \cap M_{\ul{n}}^1.$ From \eqref{eqn:M-n-1}, one has 
$K_{\ul{n}}^1 = {\rm S}({\rm O}(2) \times \cdots \times {\rm O}(2))$. Its connected component of the identity is 
$K_{\ul{n}}^{\circ} = \SO(2)\times \cdots \times \SO(2)$ which is the maximal torus of the compact group $K^1 = \SO(n)$. Hence, 
$$
H^q(\mathfrak{g}^1, K^1, \, \J_{\lambda}^1 \otimes \M_{\lambda}^1) \ \cong \ 
\Hom_{K_{\ul{n}}^1} (\wedge^q\mathfrak{p}^1_c, \, \chi_{\lambda}^1 \otimes \M_{\lambda}^1) \ \subset \ 
\Hom_{K_{\ul{n}}^\circ} (\wedge^q\mathfrak{p}^1_c, \, \chi_{\lambda}^1 \otimes \M_{\lambda}^1).
$$
The inclusion is in fact an equality as will become clear by the end of this discussion. As a $K_{\ul{n}}^\circ$-module, we have a splitting:
$\mathfrak{p}^1_c = \a_{\ul{n},c}^1 \oplus \u_{\ul{n},c}$, with the adjoint action of $K_{\ul{n}}^\circ$ on $\a_{\ul{n},c}^1$ being trivial. 
It is helpful to note that their dimensions are 
$\dim(\p^1_c) = n(n-1)/2 + (n-1), \ \dim(\a_{\ul{n},c}^1) = \lfloor (n-1)/2 \rfloor, \  
\dim(\u_{\ul{n},c}) = \lfloor n^2/2 \rfloor.$ Furthermore, $b_n = \dim(\u_{\ul{n},c})/2.$
\begin{equation}
\label{eqn:kunneth-frobenius-2}
H^q(\mathfrak{g}^1, K^1, \, \J_{\lambda}^1 \otimes \M_{\lambda}^1) \ \subset \
\bigoplus_{s+t = q} \wedge^s (\a_{\ul{n},c}^{1*}) \otimes \Hom_{K_{\ul{n}}^\circ} (\wedge^t \u_{\ul{n},c}, \, \chi_{\lambda}^1 \otimes \M_{\lambda}^1).
\end{equation}

\medskip
\paragraph{\it Relevant $K_{\ul{n}}^\circ$-types of $\chi_\lambda^{1} \otimes \M_\lambda^1$}

From \eqref{eqn:kunneth-frobenius-2} and the statement about dimension in Lem.\,\ref{lem:dimension-slnr} we first of all deduce 
that the inclusion in \eqref{eqn:kunneth-frobenius-2} is an equality and furthermore we deduce the following:  

\begin{lemma}
\label{lem:canonical-character-slnr}
$$
\Hom_{K_{\ul{n}}^\circ} (\wedge^t \u_{\ul{n},c}, \, \chi_{\lambda}^1 \otimes \M_{\lambda}^1)
\ = \ 
\begin{cases}
\C & \mbox{if $t = \lfloor n^2/4 \rfloor = b_n$,} \\
0 & \mbox{otherwise}, 
\end{cases}
$$
i.e., there is a unique character of the compact torus $K_{\ul{n}}^\circ$ denoted, say, $\chi_0$, that appears in $\wedge^\bullet \u_{\ul{n},c}$ and in 
$\chi^1_{\blambda} \otimes \M^1_{\blambda}$;  necessarily, it appears in a unique degree; this degree is $\bullet = b_n = \dim(\u_{\ul{n},c})/2$; 
furthermore, it appears with multiplicity one in $\wedge^{\lfloor n^2/4 \rfloor}\u_{\ul{n},c}$ and in 
$\chi^1_{\blambda} \otimes \M^1_{\blambda}$. 
\end{lemma}

\medskip
\paragraph{\it Mutatis mutandis for the main result on Lefschetz numbers for $\SL_n(\R)$}

After the above discussion, and an application of Frobenius reciprocity, we have
$$
H^q(\mathfrak{g}^1, K^1, \, \J_{\lambda}^1 \otimes \M_{\lambda}^1) \ \cong \ 
\Hom_{K_{\ul{n}}^\circ} (\wedge^q\mathfrak{p}^1_c, \, \chi_{\lambda}^1 \otimes \M_{\lambda}^1)
 \ \cong \ 
\Hom_{K^1} (\wedge^q\mathfrak{p}^1_c, \, \Ind_{K_{\ul{n}}^\circ}^{K^1}(\chi_{\lambda}^1 \otimes \M_{\lambda}^1)).
$$
From here on the reasoning is identical to the case of $\SL_n(\C)$ as in \ref{sec:slnc}; one sees that any $K^1$ equivariant map as above has to factor as
$$
\wedge^q \p^1_c \ \twoheadrightarrow \ \pi_{\chi_0} \ \hookrightarrow \ \Ind_{K_{\ul{n}}^\circ}^{K^1}(\chi_{\lambda}^1 \otimes \M_{\lambda}^1), 
$$
where $\pi_{\chi_0}$ is the irreducible representation of $K^1$ of highest weight $\chi_0.$ Furthermore, $\pi_{\chi_0}$ appears in 
$\J_{\lambda}^1 \otimes\mathcal{M}_{\lambda}^1$ with multiplicity one. (For the reader familiar with Vogan--Zuckerman \cite{vogan-zuckerman}, our $\chi_0$ is their $2\rho(\u \cap \p)$, and our $\pi_{\chi_0}$ is their 
$\mu(\mathfrak{q})$; see \cite[(2.4)]{vogan-zuckerman}. The reader may see the first author's doctoral 
dissertation\footnote{{\SMALL Nasit Darshan, {\it Cuspidal cohomology for $\GL(N)$ over number fields,} Ph.D.\ thesis (2024), IISER Pune, India.}}
for more details.) 

\begin{prop}
\label{prop:slnr-lefschetz}
There is a nonzero constant $c_{\chi_0}$ such that 
$$
\operatorname{Lef}(\theta, \SL_n(\mathbb{R}), \J_{\lambda}^1 \otimes \M_{\lambda}^1) \ = \ 
\begin{cases}
c_{\chi_0}2^{m-1}\neq 0 & n=2m, \\ c_{\chi_0}2^m \neq 0 & n=2m+1.
\end{cases}
$$
\end{prop}

The proof is exactly as in the proof of Prop.\,\ref{prop:slnc-lefschetz}.

\medskip
\subsubsection{\bf Applying Borel--Labesse--Schwermer to $\SL_n$ over $F$}
\label{sec:slnF}
We revert to the global notations as in Sect.\,\ref{sec:main_theorem_statement}. 
Recall the hypothesis of Thm.\,\ref{thm:main-sln} that $F$ is a number field which is Galois over its maximal totally real subfield $F_0$. 
Recall also from Sect.\,\ref{sec:base-field} that $F$ is contained in $\C$. 
Let $\c$ denote complex conjugation as an element of $\Gal(F/F_0).$ (If $F$ is not assumed to be a subfield of $\C$, then we can fix an embedding of 
$F$ to $\C$, and borrow complex conjugation via this embedding, and then show that the entire discussion is independent of this choice of embedding; such an 
exercise in book-keeping with added baggage of notation does not add any insight into the problem of nonvanishing of cuspidal cohomology.) 
Then $\c$ gives a rational automorphism of $G^1 = \Res_{F/\Q}(\SL_n/F)$; see, for example, 
Platonov--Rapinchuk \cite[Chap.\,2]{platonov-rapinchuk}. Let $\vartheta$ be the automorphism of $\SL_n/F$ defined as $\vartheta(g) = {}^t g^{-1}.$ 
Consider the automorphism 
$$
\alpha \ := \ \Res_{F/\Q}(\vartheta) \circ \c
$$
of $\Res_{F/\Q}(\SL_n/F)$. We contend that this $\alpha$ can be used in Thm.\,\ref{thm:bls}.

\medskip
\paragraph{\it The action of $\c$ on $G^1(\R)$}
\label{sec:action-of-c}
To begin, one needs to understand the action of complex conjugation as a Galois element $\c \in \Gal(F/F_0)$ on $G^1(\R)$.
The first step is to understand the action of $\c$ on $F \otimes \R$, which naturally falls into two cases depending on whether $F$ contains a CM subfield or not. 

\medskip
{\bf Case - 1}: $F$ is in the {\bf CM}-case, by which we mean $F$ contains a CM subfield; in this case, $F_0$ is the largest totally real subfield of $F$, and $F_1$ is a totally imaginary quadratic extension of $F_0$ contained in $F$; then $F_1$ is the largest CM subfield of $F$. Say, $[F_0 : \Q] = d$ and $[F:F_1] = k.$ Fix $\sigma \in \Sigma_{F_0}$; 
suppose $\{\nu, \bar\nu\} \subset \Sigma_{F_1}$ are the conjugate-embeddings of $F_1$ that restrict to $\sigma$, and furthermore, suppose, $\{\eta_1,\dots,\eta_k\}$ are all the embeddings in $\Sigma_F$ that restrict to $\nu$; necessarily, $\{\bar\eta_1,\dots, \bar\eta_k\}$ are all the embeddings that restrict to $\bar\nu$ on $F_1.$
\begin{claim}
Thinking of $\c$ as an element of $\Gal(F/F_0)$ one has $\eta_j \circ \c \in \{\bar\eta_1, \dots, \bar\eta_k\}.$
\end{claim}
The restriction of $\eta_j \circ \c$ to $F_0$ is $\sigma$; hence $\eta_j \circ \c \in \{\eta_1, \dots, \eta_k, \, \bar\eta_1, \dots, \bar\eta_k\}.$ If $\eta_j \circ \c = \eta_{j'}$ 
then restricting to $F_1$ we get $\nu(\c(x)) = \nu(x)$ for all $x \in F_1$ which is not possible; proving the claim. There is a permutation $\varpi_\sigma$ of 
$\{1,2,\dots,k\}$ such that $\eta_j \circ \c = \bar\eta_{\varpi_\sigma(j)}.$ Suppose we write 
$$
F \otimes_\Q \R \ = \  \prod_{\sigma \in \Sigma_{F_0}} F \otimes_{F_0, \sigma} \R \ \simeq \ \prod_{\sigma \in \Sigma_{F_0}} (\C \times \cdots \times \C),
$$
where one has $k$-factors of $\C$ above $\sigma$, with the isomorphisms of these $k$ completions of $F$ with $\C$ corresponding to $\eta_j$'s as being the distinguished embeddings for the pairs $\{\eta_j, \bar\eta_j\}$ of conjugate embeddings. If we write $\ul{z} \in F \otimes_\Q \R,$ with 
$\ul{z} = (\ul{z}^\sigma)_{\sigma \in \Sigma_{F_0}}$ and $\ul{z}^\sigma = (z^\sigma_1,\dots, z^\sigma_k)$ then the action of $\c$ as $\c \otimes 1$ on $\ul{z}$ is: 
$$
\c(\ul{z}) = (\c(\ul{z}^\sigma))_{\sigma \in \Sigma_{F_0}}; \quad 
\c(\ul{z}^\sigma) = \c(z^\sigma_1,\dots,z^\sigma_k) = (\bar{z}^\sigma_{\varpi_\sigma(1)},\dots, \bar{z}^\sigma_{\varpi_\sigma(k)} ), 
$$
i.e., $\c$ permutes and conjugates the entries above a given $\sigma.$ Similarly, if we write 
$$
G^1(\R) = \SL_n(F \otimes_\Q \R) = \prod_{\sigma \in \Sigma_{F_0}} \SL_n(F \otimes_{F_0, \sigma} \R ) 
 \ \simeq \ \prod_{\sigma \in \Sigma_{F_0}} (\SL_n(\C) \times \cdots \times \SL_n(\C)), 
$$
then the action of $\c$ on $G^1(\R)$ is via permuting and complex-conjugating the matrices in the $k$-fold product $\SL_n(\C) \times \cdots \times \SL_n(\C)$ 
above a given $\sigma.$

\medskip
{\bf Case - 2:} $F$ is in the {\bf TR}-case, by which we mean $F$ does not contain a CM subfield; we still have $F_0$ is the largest totally real subfield of $F$. 
Say, $[F_0 : \Q] = d$ and $[F:F_0] = k.$ Fix $\sigma \in \Sigma_{F_0}$; suppose 
$\rho_1,\dots, \rho_r, \, \nu_1, \dots, \nu_s, \, \bar{\nu}_1, \dots \bar{\nu}_s,$
are all the embeddings of $F$ into $\C$ that restrict to $\sigma$ on $F_0$; where, $\rho_i$'s are real embeddings, and $\nu_j,  \bar\nu_j$ are pairs of conjugate
embeddings;  of course $k = r + 2s.$ Clearly, $\rho_i \circ \c$ is again a real embedding, hence pre-composing by $\c$ permutes $\rho_1,\dots, \rho_r.$ Consider 
a pair $\{\eta, \bar\eta\}$ of conjugate embeddings over $\sigma$, then on pre-composing by $\c$ we have one of the following two cases:
	\begin{enumerate}
		\item[(i)] $\{\eta \circ \c, \bar\eta \circ \c\} = \{\eta, \bar\eta\}$; 
		\item[(ii)] $\{\eta \circ \c, \bar\eta \circ \c\} \cap \{\eta, \bar\eta\} = \emptyset$. 
	\end{enumerate}
In (i), necessarily, $\eta \circ \c = \bar\eta$. Say, (i) happens to $a$ many pairs of conjugate embeddings over a fixed $\sigma$, and 
(ii) happens to $2b$ many pairs of conjugate embeddings over a fixed $\sigma$; of course, $s = a + 2b$. Recall, that for any pair $\{\eta, \bar\eta\}$ corresponding 
to a complex place, the completion of $F$ at this place is identified with $\C$ via a non canonically chosen distinguished element in this pair. In (ii), arbitrarily fix an $\eta$ as the distinguished embedding in the pair $\{\eta,\bar\eta\}$; now fix $\bar\eta \circ \c$ as the distinguished element in the pair $\{\eta \circ \c, \bar\eta \circ \c\};$ this way of choosing distinguished elements is well-defined since $\overline{\bar\eta \circ \c} \circ \c = \eta.$ Now one has:
$F \otimes_\Q \R = \prod_{\sigma \in \Sigma_{F_0}} F \otimes_{F_0, \sigma} \R$, where 
$$
F \otimes_{F_0, \sigma} \R \ \simeq \ 
\R \times \cdots \times \R 
\ \times \ \C \times \cdots \times \C \ 
\ \times \ (\C \times \C) \times \cdots \times (\C \times \C),
$$
where there are $r$ factors of $\R,$ followed by $s$ factors of $\C$ grouped as $a$ copies of $\C$ followed by $b$ copies of $\C \times \C$. If 
$\xi \in F \otimes_{F_0, \sigma} \R$ is written as 
$$
\xi = (x_1,\dots, x_r; \, z_1,\dots,z_a; \, u_1, w_1, \dots, u_b, w_b)
$$
with $x_i \in \R$, $z_j \in \C$, and $u_l, w_l \in \C$, then the action of $\c$ on $\xi$ permutes the real entries, conjugates the $z_j$'s, and interchanges and conjugates the pairs $u_l, w_l$, i.e., 
$$
(\c \otimes 1)(\xi) = (x_{1'},\dots, x_{r'}; \, \bar{z}_1,\dots, \bar{z}_a; \,  \bar{w}_1, \bar{u}_1, \dots, \bar{w}_b, \bar{u}_b).
$$
The point relevant in the proof below is that that real entries are permuted and the complex entries are permuted and conjugated. 
This description applies just the same to elements of $G^1(\R) = \SL_n(F \otimes \R) = \prod_{\sigma \in \Sigma_{F_0}} \SL_n(F \otimes_{F_0, \sigma} \R)$; the action of 
$\c$ stabilizes $\SL_n(F \otimes_{F_0, \sigma} \R)$, and for an element  
$\ul{g} \in \SL_n(F \otimes_{F_0, \sigma} \R)$ written as 
$\ul{g} = (t_1,\dots, t_r; \, g_1, \dots, g_s)$
with $t_i \in \SL_n(\R)$ and $g_j \in \SL_n(\C)$, the action of $\c$ on $\ul{g}$ permutes the real matrices $t_i$, and permutes and conjugates the complex matrices $g_j.$

\medskip
\paragraph{\it The action of $\alpha$ on the coefficient system}

The following lemma needs strong-purity of the highest weight of the coefficient system. 

\begin{lemma}
\label{lem:alpha-M-lambda}
Let $\lambda \in X^+_{00}(\Res_{F/\Q}(T_0) \times E)$ be a strongly-pure weight, $\iota : E \to \C$ an embedding, and 
$(\rho_{{}^\iota\!\lambda},  \M_{{}^\iota\!\lambda})$ the representation of 
$G(\Q) = \GL_n(F)$ of highest weight ${}^\iota\lambda$. Its restriction to $G^1(\Q)$, 
is invariant under 
$\alpha$, i.e., ${}^\alpha\rho_{{}^\iota\!\lambda^1} \simeq \rho_{{}^\iota\!\lambda^1}.$
\end{lemma}

\begin{proof}
For $g \in G(\Q)$, note that ${}^\alpha\rho_{{}^\iota\!\lambda}(g) = \rho_{{}^\iota\!\lambda}(\alpha(g))$. Keeping in mind 
$\M_{{}^\iota\!\lambda} = \otimes_{\eta : F \to \C} \M_{{}^\iota\!\lambda^\eta}$ where ${}^\iota\!\lambda^\eta = \lambda^{\iota^{-1}\circ \eta}$
(see \ref{sec:the-sheaf} and \ref{sec:strong-pure-C}), one has
\begin{multline*}
{}^\alpha\rho_{{}^\iota\!\lambda}(g) =  \rho_{{}^\iota\!\lambda}(\alpha(g)) = \bigotimes_{\eta : F \to \C} \rho_{{}^\iota\!\lambda^\eta}(\eta(\alpha(g))) 
= \bigotimes_{\eta : F \to \C} \rho_{{}^\iota\!\lambda^\eta}(\bar\eta(\, {}^t \!g^{-1})) \\ 
= \bigotimes_{\eta : F \to \C} \rho_{{}^\iota\!\lambda^{\bar\eta}}(\eta(\, {}^t \!g^{-1}))  
\simeq \bigotimes_{\eta : F \to \C} \rho_{{}^\iota\!\lambda^{\bar\eta}}^\v(\eta(g)) \simeq \bigotimes_{\eta : F \to \C} \rho_{-w_0(\lambda^{\iota^{-1} \circ \bar\eta})}(\eta(g)).
\end{multline*}
Invoking strong-purity ((iii) of Prop.\,\ref{prop:strong-pure-E}), one has that the restriction to the maximal torus $T_0^1$ of $\SL(n)$ of the weight 
$-w_0(\lambda^{\iota^{-1} \circ \bar\eta})$ is the same as the restriction of $\lambda^{\iota^{-1} \circ \eta}$. Thus, for $g \in G^1(\Q)$, we have 
$
{}^\alpha\rho_{{}^\iota\!\lambda^1}(g) \simeq \bigotimes_{\eta : F \to \C} \rho_{\lambda^{\iota^{-1} \circ \eta}}(\eta(g)) = \rho_{{}^\iota\!\lambda^1}(g).$
\end{proof}

\medskip
\paragraph{\it The action of $\alpha$ on an irreducible unitary representation $\pi_\infty$}
To apply Thm.\,\ref{thm:bls} one needs to show the existence of an irreducible 
admissible unitary representation $\pi_\infty$ of $G^1(\R)$ such that the Lefschetz number $\Lef(\alpha, G^1(\R); \, \pi_\infty \otimes \M^1_{{}^\iota\lambda, \C})$ is nonzero. 
For a complex place $v$, one has a pair $(\eta_v, \bar\eta_v)$ of conjugate embeddings. If $F$ is in the {\bf TR}-case then the distinguished element in this pair 
is  chosen as in \ref{sec:action-of-c}. The finite-dimensional representation $\M^1_{{}^\iota\lambda_v}$ is 
the restriction to $\SL_n(F_v) \cong \SL_n(\C)$ of the 
representation $\M_{{}^\iota\!\lambda^{\eta_v}} \otimes \M_{{}^\iota\!\lambda^{\bar\eta_v}}$ of $\GL_n(F_v) \cong \GL_n(\C)$. 
Consider the pair 
${\blambda}_v  \ = \ ({}^\iota\!\lambda^{\eta_v}, {}^\iota\!\lambda^{\bar\eta_v})$ and the unitary representation 
$\pi_v = \J^1_{\blambda_v}$ 
 of $\SL_n(F_v) = \SL_n(\C)$ as in \ref{par:unitary-rep}.  
From Lem.\,\ref{lem:I_infty} one has ${}^\theta \J^1_{\chi_{{\blambda}_v}} \simeq \J^1_{\chi_{{\blambda}_v}}$ for the Cartan involution $\theta$. 
Prop.\,\ref{prop:slnc-lefschetz} gives 
$\Lef(\theta, \SL_n(\C), \pi_v \otimes \M^1_{\blambda_v}) \neq 0.$ 
Similarly, for a real place $v$ of $F$ corresponding to a real embedding $\eta_v$, if $\lambda_v = {}^\iota\!\lambda^{\eta_v}$ and putting 
$\pi_v = \J^1_{\lambda_v}$, from Prop.\,\ref{prop:slnr-lefschetz} one has $\Lef(\theta, \SL_n(\R), \pi_v \otimes \M^1_{\lambda_v}) \neq 0$
for the Cartan involution $\theta$. 
Consider the representation $\pi_\infty = \otimes_{v \in S_\infty} \pi_v$, with $\pi_v$'s as above. Analogous to Lem.\,\ref{lem:alpha-M-lambda}, the following lemma
is a consequence of strong-purity and the action of $\alpha$ on $G^1(\R)$.

\begin{lemma}
\label{lem:alpha-pi-infty}
Let $\lambda \in X^+_{00}(\Res_{F/\Q}(T_0) \times E)$ be a strongly-pure weight, $\iota : E \to \C$ an embedding. 
Then for the representation $\pi_\infty = \otimes_{v \in S_\infty} \pi_v$, with $\pi_v$ as above, one has ${}^\alpha\pi_\infty \simeq \pi_\infty$ as an 
equivalence of $G^1(\R)$-modules.
\end{lemma}

\begin{proof}
Suppose $F$ is in the {\bf CM}-case with related notations as in \ref{sec:action-of-c}. Write $\pi_\infty = \otimes_{\sigma \in \Sigma_{F_0}} \pi(\sigma)$, with 
$\pi(\sigma) = \otimes_{v | \sigma} \pi_v$; by $v|\sigma$ we mean a place $v$ of $F$ above the place of $F_0$ corresponding to the embedding $\sigma$. 
Strong-purity (especially, Prop.\,\ref{prop:strong-pure-weights-base-change}) implies that the highest weights indexed
by the pairs $\eta_j, \bar\eta_j$ are all pair-wise equal (to the weights indexed by $\nu, \bar\nu$ of $F_1$), and so also all the 
$\pi_v$'s, for $v|\sigma$, are isomorphic to each other. From \ref{sec:action-of-c} it follows that the action of $\alpha$ on $\SL_n(F \otimes_{F_0, \sigma} \R)$ 
is via permuting the factors of $\SL_n(\C)$ followed by the Cartan involution $\theta$ which preserves each of these local representations. 

Suppose $F$ is in the {\bf TR}-case with related notations as in \ref{sec:action-of-c}. For $\sigma \in \Sigma_{F_0}$, suppose  
$v_1,\dots,v_r, w_1,\dots w_s$ are all the places of $F$ above the place of $F_0$ corresponding to $\sigma$; here $v_i$ is the real place corresponding 
to the real embedding $\rho_i$, and $w_j$ is the complex place corresponding to the pair of conjugate embeddings $\{\eta_j, \bar\eta_j\}$ 
with the distinguished element $\eta_j$ fixed as in \ref{sec:action-of-c}. From Prop.\,\ref{prop:strong-pure-weights-base-change} it follows that the highest weights indexed
by any any $\rho_i$ or $\eta_j$ or $\bar\eta_j$ are all equal; hence all the $\pi_{v_i}$ are isomorphic modules of $\SL_n(\R)$, and 
all the $\pi_{w_j}$ are isomorphic modules of $\SL_n(\C)$. The action of $\alpha$ on $\SL_n(F \otimes_{F_0, \sigma} \R)$ 
is via permuting the factors of $\SL_n(\R)$, and separately permuting the factors of $\SL_n(\C)$, followed by the Cartan involution on each factor, which 
preserves each of these local representations. 
\end{proof}

\medskip
\paragraph{\it Conclusion of proof}

The final step in the proof is to show the following
\begin{prop}
\label{prop:lefschetz}
$$
\Lef(\alpha, G^1(\R); \, \pi_\infty \otimes \M^1_{{}^\iota\lambda, \C}) \ \neq \ 0.
$$
\end{prop}

\begin{proof}
To begin, we have 
$$
\Lef(\alpha, G^1(\R); \, \pi_\infty \otimes \M^1_{{}^\iota\lambda, \C}) \ = \ 
\prod_{\sigma \in \Sigma_{F_0}} \Lef(\alpha, \SL_n(F \otimes _{F_0, \sigma} \R); \, \pi(\sigma) \otimes \M(\sigma)),
$$
where $\M(\sigma) := \otimes_{\eta|_{F_0} = \sigma} \M_{{}^\iota\!\lambda^\eta}$, and $\pi(\sigma) = \otimes_{v | \sigma} \pi_v$ as in the proof of 
Lem.\,\ref{lem:alpha-pi-infty}. It suffices to show $\Lef(\alpha, \SL_n(F \otimes _{F_0, \sigma} \R); \, \pi(\sigma) \otimes \M(\sigma)) \neq 0.$  

\medskip
Towards this, let us digress for the moment, and consider the underlying linear-algebraic context: suppose $A$ is a graded vector space, which, up to 
a shift in degree, is the exterior algebra of a finite-dimensional vector space, and is of the form:
$$
A \ = \ \bigoplus_{j = b}^t A_j, \quad A_j = \wedge^{j-b} (\C^{t-b}), 
$$
for integers $0 < b < t$. Let $d_j := \dim(A_j)$. 
Suppose $f : A \to A$ is a graded linear map, $f = \oplus f_j$, with $f_j : A_j \to A_j$; suppose there is a nonzero scalar $c$ independent of $j$ such that:
\begin{equation}
\label{eqn:cond-f-j}
f_j = (-1)^j c 1_{A_j}.
\end{equation}
Fix a positive integer $k$, and let $\cA = A^{\otimes^k} = A \otimes \cdots \otimes A$ be the $k$-fold tensor product of $A$ with itself. It is a graded vector space
$$
\cA = \bigoplus_{m = bk}^{tk} \cA_m
$$
where the degree $m$ summand is described via K\"unneth:
\begin{equation}\label{eqn:kunneth-cA}
\cA_m \ = \ \bigoplus_{\substack{b \leq i_1,\dots, i_k \leq t \\ \sum_j i_j = m}} A_{i_1} \otimes \cdots \otimes A_{i_k}.
\end{equation}
Denote the summand $A_{i_1} \otimes \cdots \otimes A_{i_k}$ as $\cA_m(i_1,\dots,i_k).$
Fix a permutation $\varpi$ of $\{1,2,\dots,k\}$; assume $\varpi^2$ is trivial; then there exist non-negative integers $p$ and $q$ with $k = 2p + q$ such that $\varpi$ is a 
product of $p$ transpositions; up to relabelling, say, 
$\varpi \ = \ (1,2)(3,4)\cdots (2p-1,2p).$
The graded map $f : A \to A$ and $\varpi$ determine a graded linear map $\cF : \cA \to \cA$ defined as 
$$
\cF \:= \ (f \otimes \cdots \otimes f) \circ \varpi.
$$
Suppose $\cF = \oplus \cF_m$, with $\cF_m : \cA_m \to \cA_m$, then $\cF_m$ maps the summand $\cA_m(i_1,\dots,i_k)$ to $\cA_m(i_{\varpi(1)},\dots, i_{\varpi(k)}),$
and on a pure tensor in this summand one has:
$$
\cF_m(a_{i_1}\otimes \cdots \otimes a_{i_k}) = f(a_{i_{\varpi(1)}}) \otimes \cdots \otimes f(a_{i_{\varpi(k)}}). 
$$
\begin{prop}
\label{prop:abstract-lefschetz}
Under the conditions as above on $\cF$ its Lefschetz number  is nonzero:
$$
\Lef(\cF) \ := \ \sum_m (-1)^m{\rm Trace}(\cF_m) \neq 0.
$$
\end{prop}
\begin{proof}
We begin with the following easy lemma: 
\begin{lemma}
\label{lem:lin-alg}
Suppose $V$ and $W$ are finite-dimensional vector spaces over $\C$, 
and $T \in {\rm End}(V)$ and $S \in {\rm End}(W)$. Consider $T \otimes S \in {\rm End}(V \otimes W)$, and $T^{(2)} \in {\rm End}(V \otimes V)$ which is defined as 
$T \otimes T$ followed by interchanging the factors: $T^{(2)}(v \otimes v') = T(v') \otimes T(v)$. Then:
\begin{enumerate}
\item[(i)] ${\rm Trace}(T \otimes S) = {\rm Trace}(T) \cdot {\rm Trace}(S).$ 
\item[(ii)] ${\rm Trace}(T^{(2)}) = {\rm Trace}(T^2).$
\end{enumerate}
\end{lemma}

\begin{proof}
Statement (i) is well-known. For (ii), suppose $\{v_1,\dots,v_n\}$ is an ordered basis of $V$, and suppose $T(v_i) = \sum_j a_{ji} v_j$. Let $A = [a_{ij}]$ be 
the matrix of $T$ relative to this basis. Then 
$T^{(2)}(v_i \otimes v_r) = T(v_r) \otimes T(v_i) =  (\sum_s a_{sr} v_s) \otimes (\sum_j a_{ji} v_j) = \sum_s \sum_j a_{sr} a_{ji} (v_s \otimes  v_j).$ The diagonal 
contribution for computing trace is obtained by putting $s = i$ and $j = r$ giving $a_{ir} a_{ri}$. Hence, ${\rm Trace}(T^{(2)}) = \sum_i\sum_r a_{ir} a_{ri} = 
\sum_i (A^2)_{ii} = {\rm Trace}(A^2) = {\rm Trace}(T^2).$
\end{proof}

To compute the trace of $\cF_m$, since it maps $\cA_m(i_1,\dots,i_k)$ to $\cA_m(i_{\varpi(1)},\dots, i_{\varpi(k)}),$ the only summands that are relevant are 
those that satisfy: $i_j = i_{\varpi(j)}$ for all $1 \leq j \leq k$; these are the diagonal blocks for $\cF_m$. 
Hence the relevant indices $(i_1,\dots,i_k)$ satisfies these conditions: 
\begin{equation}
\label{eqn:relevant-index}
b \leq i_1,\dots,i_k \leq t, \quad \sum_j i_j = m, \quad i_1 = i_2, \ i_3 = i_4, \ \dots, \ i_{2p-1} = i_{2p}. 
\end{equation}
Let $\cI_m$ denote the set of all such relevant indices. For brevity, let $\cF_m(i_1,\dots,i_k)$ denote the $\cF_m$ restricted to $\cA_m(i_1,\dots,i_k)$.
For a relevant index we have $\cF_m$ maps $\cA_m(i_1,\dots,i_k)$ into itself; if we group the factors of $\cA_m(i_1,\dots,i_k)$ as
$$
\cA_m(i_1,\dots,i_k) \ = \ (A_{i_1} \otimes A_{i_1}) \otimes (A_{i_3} \otimes A_{i_3}) \otimes \cdots \otimes (A_{i_{2p-1}} \otimes A_{i_{2p}}) \otimes 
A_{i_{2p+1}} \otimes \cdots \otimes A_{k}, 
$$
then, on this $\cF_m$ is given by
$$
\cF_m(i_1,\dots,i_k) \ = \ f_{i_1}^{(2)} \otimes f_{i_3}^{(2)} \otimes \cdots \otimes f_{i_{2p-1}}^{(2)} \otimes f_{i_{2p+1}} \otimes \cdots \otimes f_{i_k}.
$$
Hence the trace of $\cF_m(i_1,\dots,i_k)$ is computed by a repeated application of Lem.\,\ref{lem:lin-alg}: 
$$
{\rm Trace}(\cF_m(i_1,\dots,i_k)) \ = \ {\rm Trace}(f_{i_1}^2) \cdot  {\rm Trace}(f_{i_3}^2) \cdots {\rm Trace}(f_{i_{2p-1}}^2) \cdot 
{\rm Trace}(f_{i_{2p+1}}) \cdots {\rm Trace}(f_{i_k}).
$$
From the conditions imposed on $f_j$, this implies:
$$
{\rm Trace}(\cF_m(i_1,\dots,i_k)) \ = \ c^2d_{i_1} \cdot  c^2d_{i_3} \cdots c^2d_{i_{2p-1}} \cdot 
\prod_{j=2p+1}^k (-1)^{i_j}\cdot c \cdot d_{i_j}
\ = \ (-1)^m c^k \cdot d_{I}
$$
where $d_I := d_{i_1} \cdot  d_{i_3} \cdots d_{i_{2p-1}} \cdot \prod_{j=2p+1}^k d_{i_j}$ is the product of the dimensions that appear; the point to note is 
that $d_I$ is a positive integer; we have also used $(c^2)^p \cdot c^q = c^{2p+q} = c^k$, and the parity condition $m \equiv i_{2p+1} + i_{2p+2} + \dots + i_k \pmod{2}$ 
which follows from the conditions on a relevant index in \eqref{eqn:relevant-index}. Hence, 
${\rm Trace}(\cF_m) \ = \ (-1)^m c^k (\sum_{I \in \cI_m} d_I),$ whence, 
$$
\Lef(\cF) = \sum_m (-1)^m {\rm Trace}(\cF_m) = c^k \left(\sum_m \sum_{I \in \cI_m} d_I \right) \neq 0.
$$
This concludes the proof of Prop.\,\ref{prop:abstract-lefschetz}.
\end{proof}

This also concludes the linear-algebraic digression. 
Getting back to the proof of Prop.\,\ref{prop:lefschetz}, recall that one 
needs to show $\Lef(\alpha, \SL_n(F \otimes _{F_0, \sigma} \R); \, \pi(\sigma) \otimes \M(\sigma)) \neq 0.$

\medskip

Suppose $F$ is in the {\bf CM}-case with related notations as in \ref{sec:action-of-c}. The linear-algebraic context may be adumbrated as follows: 
\begin{itemize}
\item $A$ is the $\SL_n(\C)$ cohomology of the unitary representation constructed in \ref{sec:slnc}; 
\item $b = n(n-1)/2$ the bottom-degree, and $t = b + n-1$ the top-degree; see Lem.\,\ref{lem:coh-j-lambda}. 
\item $f : A \to A$ is the map induced in cohomology by the Cartan involution $\theta$ on $\SL_n(\C)$; for the condition \eqref{eqn:cond-f-j} on $f_j$, see the proof 
of Prop.\,\ref{prop:slnc-lefschetz}. 
\item $k = [F:F_1]$ is the number of imaginary places of $F$ sitting over a fixed $\sigma : F_0 \to \R$; 
\item $\cA$ is the $\SL_n(\C) \times \cdots \times \SL_n(\C)$ cohomology $A \otimes \cdots \otimes A$ computed via K\"unneth; 
\item $\varpi$ is the permutation of the $k$ imaginary places of $F$ induced by $\c \in \Gal(F/F_0)$.
\end{itemize}
Then $\cF$ is exactly the map induced in the $\SL_n(F \otimes _{F_0, \sigma} \R)$-cohomology of $\pi(\sigma) \otimes \M(\sigma)$ by algebraic automorphism $\alpha$. 
It follows from Prop.\,\ref{prop:abstract-lefschetz} that 
$$\Lef(\alpha, \SL_n(F \otimes _{F_0, \sigma} \R); \, \pi(\sigma) \otimes \M(\sigma)) = \Lef(\cF) \neq 0.$$

\medskip

Suppose $F$ is in the {\bf TR}-case with related notations as in \ref{sec:action-of-c}. Fix $\sigma : F_0 \to \R$. Suppose $v_1,\dots,v_r$ (resp., $w_1,\dots,w_s$)
are the real (resp., imaginary) places of $F$ above $\sigma.$ Let 
$\cS_r(\sigma) = \prod_{i=1}^r \SL_n(F_{v_j}) = \prod_{i=1}^r \SL_n(\R),$ and  $\cS_c(\sigma) = \prod_{j=1}^s \SL_n(F_{w_j}) \simeq 
\SL_n(\C) \times \cdots \SL_n(\C)$, with the distinguished embeddings chosen as in \ref{sec:action-of-c}. Then 
$\SL_n(F \otimes _{F_0, \sigma} \R) = \cS_r(\sigma) \times \cS_c(\sigma).$  
Correspondingly, let $\pi_r(\sigma) = \otimes _{i=1}^r \pi_{v_i}$ and $\pi_c(\sigma) = \otimes _{j=1}^r \pi_{w_j}$; similarly, for $\M_r$ and $\M_c$. 
Recall that the real and complex copies are permuted separately by $\c$; the action of $\alpha$ on $\cS_r(\sigma)$ and $\cS_c(\sigma)$ are as in \ref{sec:action-of-c}. We have
\begin{multline*}
\Lef(\alpha, \SL_n(F \otimes _{F_0, \sigma} \R); \, \pi(\sigma) \otimes \M(\sigma)) = \\
\Lef(\alpha, \cS_r(\sigma); \pi_r(\sigma) \otimes \M_r(\sigma)) \cdot \Lef(\alpha, \cS_c(\sigma); \pi_c(\sigma) \otimes \M_c(\sigma)).
\end{multline*}
To show $\Lef(\alpha, \cS_c(\sigma); \pi_c(\sigma) \otimes \M_c(\sigma)) \neq 0$ we can set up the linear algebraic context exactly as in the {\bf CM}-case. 
To show 
$\Lef(\alpha, \cS_r(\sigma); \pi_r(\sigma) \otimes \M_r(\sigma)) \neq 0$, similarly, take:   
\begin{itemize}
\item $A$ to be the $\SL_n(\R)$ cohomology of the unitary representation constructed in \ref{sec:slnr}; 
\item $b = \lfloor n^2/4 \rfloor$ and $t = b + \lfloor (n-1)/2 \rfloor$; see Lem.\,\ref{lem:dimension-slnr}.
\item $f : A \to A$ is the map induced in cohomology by the Cartan involution $\theta$ on $\SL_n(\R)$; 
\item $\cA = A^{\otimes^r} = A \otimes \cdots \otimes A$ computed via K\"unneth; 
\item $\varpi$ is the permutation of the $r$ real places of $F$ above $\sigma$ induced by $\c \in \Gal(F/F_0)$.
\end{itemize}
It follows from Prop.\,\ref{prop:abstract-lefschetz} that $\Lef(\alpha, \cS_r(\sigma); \pi_r(\sigma) \otimes \M_r(\sigma)) \neq 0$. 

\end{proof}

A simple example might help the reader when $F$ is a non-CM quadratic extension of an imaginary quadratic field $F_1$; in this case $F_0 = \Q$. 
Think of $A$ as the cohomology of $\SL_2(\C)$ and $\cA$ 
the cohomology of $\SL_2(\C) \times \SL_2(\C)$ given by $A \otimes A$; the element $\c$ interchanges and conjugates the two factors. Then, 
$A = A_1 \oplus A_2$; $d_1 = d_2 = 1;$ $k = 2$.  By K\"unneth, $\cA = \cA_2 \oplus \cA_3 \oplus \cA_4$, where $\cA_2 = A_1 \otimes A_1,$ 
$\cA_3 = A_1 \otimes A_2 \oplus A_2 \otimes A_1,$ and
$\cA_4 = A_2 \otimes A_2.$ 
The permutation $\varpi$ is the transposition $(1,2).$ The set of relevant indices are $\cI_2 = (1,1)$, $\cI_3 = \emptyset$ and $\cI_4 = (2,2)$. 
We get ${\rm Trace}(\cF_2) = c^2d_1$,  
${\rm Trace}(\cF_3) = 0,$ and ${\rm Trace}(\cF_4) = c^2d_2$; where $c$ is a nonzero scalar. 
Hence, $\Lef(\cF) = c^2(d_1 + d_2) \neq 0.$

\bigskip

{\small
{\it Acknowledgements:} We are grateful to the Institute for Advanced Study, Princeton, for a summer collaborator's grant in 2023 when much of this work was completed. 
We thank other members
of this summer collaboration: Baskar Balasubramanyam, Chandrasheel Bhagwat, P.\,Narayanan, and Freydoon Shahidi, who acted as a sounding board at various times. 
We thank David Vogan for a conversation which gave us clarity on 
tempered cohomological representations. Finally, we are deeply grateful to Frank Calegari for his interest and insightful questions on this work; 
the final steps in 6.2.3 owes a lot to his questions.

\bigskip


\end{document}